\documentclass[11pt]{article}
\usepackage{amsthm, amsmath, amssymb, amsfonts, url, booktabs, tikz, setspace, fancyhdr}
\usepackage[margin = 1in]{geometry}

\usepackage{hyperref}

\newtheorem{theorem}{Theorem}[section]
\newtheorem{proposition}[theorem]{Proposition}
\newtheorem{lemma}[theorem]{Lemma}

\newtheorem{conjecture}[theorem]{Conjecture}
\newtheorem{claim}[theorem]{Claim}

\theoremstyle{definition}

\theoremstyle{remark}

\newcommand{\bG}{\hat{G}}
\newcommand{\tG}{\tilde{G}}
\newcommand{\G}{\mathcal{G}}
\newcommand{\Gh}{\hat{\mathcal{A}}}

\newcommand{\EE}{\mathbb{E}}

\newcommand{\bm}{\mathbf{m}}

\newcommand{\cP}{\mathcal{P}}

\newcommand{\A}{\mathcal{A}}
\newcommand{\B}{\mathcal{B}}
\newcommand{\C}{\mathcal{C}}

\newcommand{\e}{\epsilon}
\renewcommand{\S}{\mathcal{S}}

\newcommand{\HSX}{\beta}
\newcommand{\HSe}{\gamma}
\newcommand{\HSep}{\gamma'}

\def\Pr{\mathop{\rm Pr}\nolimits}

\tikzstyle{p}+=[fill=black, circle, minimum width = 1pt, inner sep =
1pt]
\tikzstyle{w}+=[fill=white, draw, circle, minimum width = 1pt, inner sep =
1.5pt]

\begin{document}

\title{On the K\L R conjecture in random graphs}

\author{D. Conlon\thanks{Mathematical Institute, Oxford OX1 3LB, United
Kingdom.
E-mail: {\tt
david.conlon@maths.ox.ac.uk}. Research supported by a Royal Society University Research Fellowship.} \and W. T. Gowers\thanks{Department of Pure Mathematics and Mathematical Statistics, Wilberforce Road, Cambridge CB3 0WB, UK. E-mail: {\tt w.t.gowers@dpmms.cam.ac.uk}. Research supported by a Royal Society 2010 Anniversary Research Professorship.} \and W. Samotij\thanks{School of Mathematical Sciences, Tel Aviv University, Tel Aviv, Israel; and Trinity College, Cambridge CB2 1TQ, UK. E-mail: {\tt samotij@post.tau.ac.il}. Research supported in part by a Trinity College JRF.} \and M. Schacht\thanks{Fachbereich Mathematik, Universit\"at Hamburg, Bundesstra\ss e 55, D-20146 Hamburg, Germany. E-mail: {\tt schacht@math.uni-hamburg.de}. Research supported by the Heisenberg programme of the DFG.}}

\date{}

\maketitle

\begin{abstract}
The K\L R conjecture of Kohayakawa, \L uczak, and R\"odl  is a statement that allows one to prove that asymptotically almost surely all subgraphs of the random graph $G_{n,p}$, for sufficiently large $p : = p(n)$, satisfy an embedding lemma which complements the sparse regularity lemma of Kohayakawa and R\"odl. We prove a variant of this conjecture which is sufficient for most known applications to random graphs. In particular, our result implies a number of recent probabilistic versions, due to Conlon, Gowers, and Schacht, of classical extremal combinatorial theorems. We also discuss several further applications.
\end{abstract}

\section{Introduction}

Szemer\'edi's regularity lemma~\cite{Sz78}, which played a crucial role in Szemer\'edi's proof of the Erd\H{o}s-Tur\'an conjecture~\cite{Sz75} on long arithmetic progressions in dense subsets of the integers, is one of the most important tools in extremal graph theory (see~\cite{KSSS02, KO09, RS10}). Roughly speaking, it says that the vertex set of every graph $G$ may be divided into a bounded number of parts in such a way that most of the induced bipartite graphs between different parts are pseudorandom. 

More precisely, a bipartite graph between sets $U$ and $V$ is said to be $\e$-\emph{regular} if, for every $U' \subseteq U$ and $V' \subseteq V$ with $|U'| \geq \e |U|$ and $|V'| \geq \e |V|$, the density $d(U', V')$ of edges between $U'$ and $V'$ satisfies
\[|d(U', V') - d(U, V)| \leq \e.\]
We will say that a partition of the vertex set of a graph into $t$ pieces $V_1, \dots, V_t$ is an equipartition if, for every $1 \leq i, j \leq t$, we have the condition that $||V_i| - |V_j|| \leq 1$. We say that the partition is $\e$-regular if it is an equipartition and, for all but at most $\e t^2$ pairs $(V_i, V_j)$, the induced graph between $V_i$ and $V_j$ is $\e$-regular. Szemer\'edi's regularity lemma can be formally stated as follows.

\begin{theorem} \label{thm:reglemma}
For every $\e > 0$ and every positive integer $t_0$, there exists a positive integer $T$ such that every graph $G$ with at least $t_0$ vertices admits an $\e$-regular partition $V_1, \dots, V_t$ of its vertex set into $t_0 \leq t \leq T$ pieces. 
\end{theorem}

Often the strength of the regularity lemma lies in the fact that it may be combined with a counting or embedding lemma that tells us approximately how many copies of a particular subgraph a graph contains, in terms of the densities $d(V_i,V_j)$ arising in an $\e$-regular partition. The so-called regularity method usually works as follows. First, one applies the regularity lemma to a graph~$G$. Next, one defines an auxiliary graph $R$ whose vertices are the parts of the regular partition of $G$ one obtains, and whose edges correspond to regular pairs with non-negligible density. (For some applications we may instead take a weighted graph, where the weight of the edge between a regular pair $V_i$ and $V_j$ is the density $d(V_i,V_j)$.) If one can then find a copy of a particular subgraph $H$ in $R$, the counting lemma allows one to find many copies of $H$ in $G$. If $R$ does not contain a copy of $H$, this information can often be used to deduce some further structural properties of the graph $R$ and, thereby, the original graph $G$. Applied in this manner, the regularity and counting lemmas allow one to prove a number of well-known theorems in extremal graph theory, including the Erd\H{o}s-Stone theorem~\cite{ES46}, its stability version due to Erd{\H o}s and Simonovits~\cite{Si68}, and the graph removal lemma~\cite{ADLRY94, EFR86, Fu95, RS78}.

For sparse graphs -- that is, graphs with $n$ vertices and $o(n^2)$ edges -- the regularity lemma stated in Theorem~\ref{thm:reglemma} is vacuous, since every equipartition into a bounded number of parts is $\e$-regular for $n$ sufficiently large. It was observed independently by Kohayakawa \cite{K97} and R\"odl that the regularity lemma can nevertheless be generalized to an appropriate class of graphs with density tending to zero. Their result applies to a natural class of sparse graphs that is wide enough for the lemma to have several interesting applications. In particular, it applies to relatively dense subgraphs of random graphs -- that is, one takes a random graph $G_{n,p}$ of density $p$ and a subgraph $G$ of $G_{n,p}$ of density at least $\delta p$ (or relative density at least $\delta$ in $G_{n,p}$), where $p$ usually tends to 0, while $\delta>0$ is usually independent of the number of vertices. 

To make this precise, we say that a bipartite graph between sets $U$ and $V$ is $(\e, p)$-\emph{regular} if, for every $U' \subseteq U$ and $V' \subseteq V$ with $|U'| \geq \e |U|$ and $|V'| \geq \e |V|$, the density $d(U', V')$ of edges between $U'$ and $V'$ satisifies
\[|d(U', V') - d(U, V)| \leq \e p.\]
That is, we alter the definition of regularity so that it is relative to a particular density $p$. This density is usually comparable to the total density between $U$ and $V$. A partition of the vertex set of a graph into $t$ pieces $V_1, \dots, V_t$ is then said to be $(\e, p)$-regular if it is an equipartition and, for all but at most $\e t^2$ pairs $(V_i, V_j)$, the induced graph between $V_i$ and $V_j$ is $(\e, p)$-regular. 

The class of graphs to which the Kohayakawa-R\"odl regularity lemma applies are the so-called upper-uniform graphs~\cite{KoRo03}. Suppose that $0 < \eta \leq 1$, $D > 1$, and $0 < p \leq 1$ are given. We will say that a graph $G$ is $(\eta, p, D)$-\emph{upper-uniform} if for all disjoint subsets $U_1$ and $U_2$ with $|U_1|, |U_2| \geq \eta |V(G)|$, the density of edges between $U_1$ and $U_2$ satisfies $d(U_1, U_2) \leq D p$. This condition is satisfied for many natural classes of graphs, including all subgraphs of random and pseudorandom graphs of density $p$. The regularity lemma of Kohayakawa and R\"odl is the following.

\begin{theorem} \label{thm:sparsereg}
For every $\e, D > 0$ and every positive integer $t_0$, there exist $\eta > 0$ and a positive integer $T$ such that for every $p\in[0,1]$, every graph $G$ with at least $t_0$ vertices that is $(\eta, p, D)$-upper-uniform admits an $(\e, p)$-regular partition $V_1, \dots, V_t$ of its vertex set into $t_0 \leq t \leq T$ pieces. 
\end{theorem}

The proof of this theorem is essentially the same as the proof of the dense regularity lemma, with the upper uniformity used to ensure that the iteration terminates after a constant number of steps. We note that different versions of the result have appeared in the literature where the $(\eta, p, D)$-upper-uniformity assumption is altered \cite{ACHKRS10, L00} or dropped completely \cite{Sc11}.

As we have already mentioned above, the usefulness of the dense regularity method relies on the existence of a corresponding counting lemma. Roughly speaking, a counting lemma says that if we start with an arbitrary graph $H$ and replace its vertices by large independent sets and its edges by $\e$-regular bipartite graphs with non-negligible density, then this blown-up graph will contain roughly the expected number of copies of $H$. Here is a precise statement to that effect.

\begin{lemma}
  \label{lemma:H-count}
  For every graph $H$ with vertex set $\{1,2,\dots,k\}$ and every $\delta>0$, there exists $\e>0$ and an integer $n_0$ such that the following statement holds. Let $n\geq n_0$ and let $G$ be a graph whose vertex set is a disjoint union $V_1 \cup \ldots \cup V_k$ of sets of size $n$. Assume that for each $ij \in E(H)$, the bipartite subgraph of $G$ induced between $V_i$ and $V_j$ is $\e$-regular and has density $d_{ij}$. Then the number of $k$-tuples $(v_1,\dots,v_k)\in V_1\times\dots\times V_k$ such that $v_iv_j\in E(G)$ whenever $ij\in E(H)$ is $n^k(\prod_{ij\in E(H)}d_{ij}\pm\delta)$.
\end{lemma}

\noindent In particular, if the density $d_{ij}$ is large for every $ij\in E(H)$, then $G$ contains many copies of $H$. 

Let us define a \textit{canonical copy} of $H$ in $G$ to be a $k$-tuple as in the lemma above: that is, a $k$-tuple $(v_1,\dots,v_k)$ such that $v_i\in V_i$ for every $i\in V(H)$ and $v_iv_j\in E(G)$ for every $ij\in E(H)$. Let us also write $G(H)$ for the number of canonical copies of $H$ in $G$. (Of course, the definitions of ``canonical copy" and $G(H)$ depend not just on $G$ but also on the partition of $G$ into $V_1,\dots,V_k$, but we shall suppress this dependence in the notation.)

In order to use Theorem \ref{thm:sparsereg}, one would ideally like a statement similar to Lemma \ref{lemma:H-count} but adapted to a sparse context. For this we would have an additional parameter $p$, which can tend to zero with $n$. We would replace the densities $d_{ij}$ by $d_{ij}p$ and we would like to show that $G(H)$ is approximately $n^kp^{e(H)}(\prod_{ij\in E(H)}d_{ij}\pm\delta)$. In order to obtain this stronger conclusion (stronger because the error estimate has been multiplied by $p^{e(H)}$), we need a stronger assumption, and the natural assumption, given the statement of Theorem \ref{thm:sparsereg} (which is itself natural), is to replace $\e$-regularity by $(\e,p)$-regularity.

Of course, we cannot expect such a result if $p$ is too small. Consider the random graph $G$ obtained from $H$ by replacing each vertex of $H$ by an independent set of size~$n$ and each edge of $H$ by a random bipartite graph with $pn^2$ edges. With high probability, $G(H)$ will be about $p^{e(H)}n^{v(H)}$. Hence, if $p^{e(H)}n^{v(H)} \ll pn^2$, then one can remove all copies of $H$ from $G$ by deleting a tiny proportion of all edges. (We may additionally delete a further small proportion of edges to ensure that all bipartite graphs corresponding to the edges of $H$ have the same number of edges.) It is not hard to see that with high probability the bipartite graphs that make up the resulting graph $G'$ will be $(\e,q)$-regular for some $q=(1-o(1))p$, but that $G'$ will contain no canonical copies of $H$.

Therefore, a sparse analogue of Lemma~\ref{lemma:H-count} cannot hold if $p \le c n^{-\frac{v(H)-2}{e(H)-1}}$ for some small positive constant $c$. Note that one can replace $H$ in the above argument by an arbitrary subgraph $H' \subseteq H$, since removing all copies of $H'$ from a graph also results in a $H$-free subgraph. This observation naturally leads to the notion of $2$-\emph{density} $m_2(H)$ of a graph $H$, defined by
\[
m_2(H) = \max \left\{ \frac{e(H')-1}{v(H')-2} \colon H' \subseteq H \text{ with } v(H') \geq 3 \right\}.
\]
(We take $m_2(K_2) = \frac{1}{2}$.) With this notation, what we have just seen is that to have any chance of an appropriate analogue of Lemma~\ref{lemma:H-count} holding, we need to assume that $p \geq C n^{-1/m_2(H)}$ for some absolute constant $C > 0$. 

Unfortunately, there is a more fundamental difficulty with finding a sparse counting lemma to match a sparse regularity lemma. Instead of sparse random graphs with many vertices, one can consider blow-ups of sparse random graphs with far fewer vertices. That is, one can pick a counterexample of the kind just described but with the sets $V_i$ of size $r$ for some $r$ that is much smaller than $n$, and then one can replace each vertex of this small graph by an independent set of $n/r$ vertices to make a graph with $n$ vertices in each $V_i$. Roughly speaking, the counterexample above survives the blowing-up process, and the result is that the hoped-for sparse counting lemma is false whenever $p=o(1)$. (For more details, see \cite{GS05, KR03}.)

However, these ``block" counterexamples have a special structure, so, for $p \geq C n^{-1/m_2(H)}$, it looks plausible that graphs for which the sparse counting lemma fails should be very rare. This intuition was formalized by Kohayakawa, \L uczak, and R\"odl \cite{KLR97}, who made a conjecture that is usually known as the K\L R conjecture. Before we state it formally, let us introduce some notation.

As above, let $H$ be a graph with vertex set $\{1,2,\dots,k\}$. We denote by $\G(H, n, m, p, \e)$ the collection of all graphs $G$ obtained in the following way. The vertex set of $G$ is a disjoint union $V_1 \cup \ldots \cup V_k$ of sets of size $n$.  For each edge $ij \in E(H)$, we add to $G$ an $(\e,p)$-regular bipartite graph with $m$ edges between the pair $(V_i, V_j)$. These are the only edges of $G$. Let us also write $\G^*(H,n,m,p,\e)$ for the set of all $G\in\G(H, n, m, p, \e)$ that do not contain a canonical copy of $H$.

Since the sparse regularity lemma yields graphs with varying densities between the various pairs of vertex sets, it may seem surprising that we are restricting attention to graphs where all the densities are equal (to $m/n^2$). However, as we shall see later, it is sufficient to consider just this case. In fact, the K\L R conjecture is more specific still, since it takes all the densities to be equal to $p$. Again, it turns out that from this case one can deduce the other cases that are needed.

\begin{conjecture} \label{conj:KLR}
Let $H$ be a fixed graph and let $\beta>0$. Then there exist $C, \e > 0$ and a positive integer $n_0$ such that 
\[|\G^*(H,n,m,m/n^2,\e)| \leq \beta^m \binom{n^2}{m}^{e(H)}\]
for every $n\geq n_0$ and every $m \geq C n^{2- 1/m_2(H)}$.
\end{conjecture}

Note that $\binom{n^2}m^{e(H)}$ is the number of graphs with vertex set $V_1\cup\dots\cup V_k$ with $m$ edges between each pair $(V_i,V_j)$ when $ij\in E(H)$ and no edges otherwise. Thus, we can interpret the conjecture as follows: the probability that a random such graph belongs to the bad set $\G^*(H,n,m,m/n^2,\e)$ is at most  $\beta^m$.

The rough idea of the conjecture is that the probability that a graph is bad is so small that a simple union bound tells us that with high probability a random graph does not contain any bad graph -- which implies that we may use the sparse embedding lemma we need. In other words, $(\e,p)$-regularity on its own does not suffice, but if you know in addition that your graph is a subgraph of a sparse random graph, then with high probability it does suffice.

More precisely, let $G$ be a random graph with $N$ vertices and edge probability $p$ and let $n = \eta N$ and $m = d p n^2 \geq n$. Then the expected number of subgraphs of $G$ of the form $\mathcal{G}^*(H, n, m, p, \epsilon)$ is at most 
\[p^{m e(H)} \beta^m \binom{n^2}{m}^{e(H)} \binom{N}{n}^{v(H)} \leq p^{m e(H)} \beta^m \left(\frac{e}{d p}\right)^{m e(H)} \left(\frac{e}{\eta}\right)^{n v(H)} \leq \beta^m \left(\frac{e}{d}\right)^{m e(H)} \left(\frac{e}{\eta}\right)^{m v(H)}.\]
Therefore, choosing $\beta$ to be sufficiently small in terms of $d, \eta$, and $H$, the probability that $G$ contains a graph in $\G^*(H,n,m,p,\e)$ is very small. By summing over the possible values of $n$ and $m$, we may rule out such bad subgraphs for all $n$ and $m$ with $n \geq \eta N$ and $m \geq d p n^2$.


This does not give us a counting lemma for $(\e,p)$-regular subgraphs of $G$, but it does at least tell us that every $(\e,p)$-regular subgraph of $G$ with sufficiently dense pairs in the right places contains a canonical copy of $H$. In other words, it gives us an embedding lemma, which makes it suitable for several applications to embedding results. For example, as noted in \cite{KLR97}, it is already sufficiently strong that a straightforward application of the sparse regularity lemma then allows one to derive the following theorem, referred to as Tur\'an's theorem for random graphs, which was eventually proved in a different way by Conlon and Gowers \cite{CG12} (for strictly balanced graphs, i.e., those for which $m_2(H) > m_2(H')$ for every proper subgraph $H'$ of $H$) and, independently, Schacht \cite{S12} (see also \cite{BMS12, ST12}). We remark that this theorem was the original motivation behind Conjecture~\ref{conj:KLR} -- see Section 6 of \cite{KLR97}. Following \cite{CG12}, let us say that a graph $G$ is $(H, \e)$-{\it Tur\'an} if every subgraph of $G$ with at least 
\[\left(1 - \frac{1}{\chi(H) - 1} + \e\right) e(G)\]
edges contains a copy of $H$. Here $\chi(H)$ is the chromatic number of $H$.
 
\begin{theorem} \label{RelativeTuran}
For every $\e > 0$ and every graph $H$, there exist positive
constants $c$ and $C$ such that
\[
\lim_{n \rightarrow \infty} \mathbb{P} \big(G_{n,p} \mbox{ is $(H,\e)$-Tur{\'a}n}\big) =
\begin{cases}
0, & \text{if $p < c n^{-1/m_2(H)}$}, \\
1, & \text{if $p > C n^{-1/m_2(H)}$}.
\end{cases}
\]
\end{theorem}

The K\L R conjecture has attracted considerable attention over the past two decades and has been verified for a number of small graphs. It is straightforward to verify that it holds for all graphs $H$ that do not contain a cycle. In this case, the class $\G^*(H,n,m,p,\e)$ will be empty. The cases $H = K_3$, $K_4$, and $K_5$ were resolved in \cite{KLR96}, \cite{GPSST07}, and \cite{GSS04}, respectively. In the case when $H$ is a cycle, the conjecture was proved in~\cite{B02, GKRS07} (see also~\cite{KK97} for a slightly weaker version). Very recently, it was proved for all balanced graphs, that is, those graphs $H$ for which $m_2(H) = \frac{e(H) - 1}{v(H) - 2}$, by Balogh, Morris, and Samotij~\cite{BMS12} and by Saxton and Thomason~\cite{ST12} in full generality.

Besides implying Theorem~\ref{RelativeTuran}, Conjecture~\ref{conj:KLR} is also sufficient for transferring many other classical extremal results about graphs to subgraphs of the random graph $G_{n,p}$, including Ramsey's theorem~\cite{R30} and the Erd\H{o}s-Simonovits stability theorem~\cite{Si68}. However, there are situations where an embedding result is not enough: rather than just a single copy of $H$, one needs to know that there are many copies. That is, one needs something more like a full counting lemma. In this paper, we shall state and prove such a ``counting version" of the K\L R conjecture for subgraphs of random graphs. Later in the paper we shall give examples of classical theorems whose sparse random versions do not follow from the K\L R conjecture but do follow from our counting result.

Our main theorem is the following.

\begin{theorem} \label{thm:main}
For every graph $H$ and every $\delta, d > 0$, there exist $\e, \xi > 0$ with the following property. For every $\eta > 0$, there is a $C > 0$ such that if $p \geq C N^{-1/m_2(H)}$, then a.a.s.\ the following holds in $G_{N,p}$:
\begin{enumerate}
\item[(i)]
For every $n \geq \eta N$, $m \geq d p n^2$, and every subgraph $G$ of $G_{N,p}$ in $\G(H,n,m,p, \e)$,
\begin{equation} \label{eqn:lowercount}
G(H) \geq \xi \left(\frac{m}{n^2}\right)^{e(H)} n^{v(H)}.
\end{equation}

\item[(ii)]
Moreover, if $H$ is strictly balanced, that is, if $m_2(H) > m_2(H')$ for every proper subgraph $H'$ of $H$, then 
\begin{equation} \label{eqn:exactcount}
G(H) = (1 \pm \delta) \left(\frac{m}{n^2}\right)^{e(H)} n^{v(H)}.
\end{equation}

\end{enumerate}
\end{theorem}

Note that strictly speaking the statements above depend not just on the graph $G$ but on the partition $V_1\cup\dots\cup V_k$ that causes $G$ to belong to $\G(H,n,m,p,\e)$. Roughly speaking, (i) tells us that if $G$ contains ``many" edges in the right places, then there are ``many" copies of $H$, while (ii) tells us that the number of copies of $H$ is roughly what one would expect for a random graph with pairs of the same densities. We note that a result similar to (ii) holds for all graphs if one is willing to allow some extra logarithmic factors. We will say more about this in the concluding remarks.

The proof of part (i) employs the ideas of Schacht \cite{S12}, as modified by Samotij \cite{Sj12}, and as a result part (i) holds with probability at least $1 - \exp(-b p N^2)$ for some $b > 0$ depending on $H, \eta$, and $d$. Part (ii) is proved using the results of Conlon and Gowers \cite{CG12} and hence hold with probability at least $1 - N^{-B}$ for any fixed $B > 0$, provided that $C$ and $N$ are sufficiently large. Since part (ii) gives an upper bound as well as a lower bound, standard results on upper tail estimates imply that the result cannot hold with the same exponential probability as part (i) (see, for example, \cite{JOR04}).

We note that weaker versions of Theorem~\ref{thm:main}, applicable for larger values of $p$, may be found in some earlier papers on the K\L R conjecture and Tur\'an's theorem  \cite{GMS07, KRS04} and more recent work on sparse regularity in pseudorandom graphs \cite{CFZ12}. We also believe that a variant of part (i) of Theorem~\ref{thm:main} may be derivable from the work of Saxton and Thomason \cite{ST12}, though they have not stated it in these terms.

\subsection{Known applications}

It is not hard to show that Theorem \ref{thm:main}, like Conjecture~\ref{conj:KLR}, implies the best possible sparse random analogues of many classical theorems in extremal graph theory. In particular, it implies Theorem~\ref{RelativeTuran} above. It also implies the following sparse random version of the Erd{\H o}s-Simonovits stability theorem, which was first proved by Conlon and Gowers~\cite{CG12} for all strictly balanced graphs and later extended to general $H$ by Samotij~\cite{Sj12}, who adapted Schacht's method for this purpose.

\begin{theorem} \label{thm:stab}
For every graph $H$ and every $\delta > 0$, there exist $\e, C > 0$ such that if $p \geq C n^{-1/m_2(H)}$, then a.a.s.\ every $H$-free subgraph $G' \subseteq G_{n,p}$ with $e(G') \geq \left(1-\frac{1}{\chi(H)-1} - \e\right)\binom{n}{2}p$ may be made $(\chi(H)-1)$-partite by removing at most $\delta p n^2$ edges.
\end{theorem}

Another easy consequence of Theorem~\ref{thm:main} is the $1$-statement of the following sparse Ramsey theorem, originally proved by R{\"o}dl and Ruci\'nski~\cite{RR95} in 1995. We remark here that the $1$-statement of Theorem~\ref{thm:Ramsey-Gnp} below also extends to general hypergraphs~\cite{CG12, FRS10}. Following~\cite{CG12}, we say that a graph $G$ is $(H,r)$-\textit{Ramsey} if every $r$-colouring of the edges of $G$ yields a monochromatic copy of $H$.

\begin{theorem}
\label{thm:Ramsey-Gnp}
For every graph $H$ that is not a star forest or a path of length $3$ and for every positive integer~$r \geq 2$, there exist constants $c, C > 0$ such that 
\[
\lim_{n \rightarrow \infty} \mathbb{P} \big(G_{n,p} \mbox{ is $(H,r)$-Ramsey}\big) =
\begin{cases}
0, & \text{if $p < c n^{-1/m_2(H)}$}, \\
1, & \text{if $p > C n^{-1/m_2(H)}$}.
\end{cases}
\]
\end{theorem}

We will omit the deductions of Theorem~\ref{thm:stab} and the $1$-statements of Theorems~\ref{RelativeTuran} and~\ref{thm:Ramsey-Gnp} from Theorem \ref{thm:main}. These are fairly standard applications of the regularity method (see, for example, \cite{GS05, K97}).

\subsection{A sparse removal lemma for graphs}

The triangle removal lemma of Ruzsa and Szemer{\'e}di~\cite{RS78} states that for every $\delta > 0$ there exists an $\e > 0$ such that if $G$ is any graph on $n$ vertices that contains at most $\e n^3$ triangles, then $G$ may be made triangle-free by removing at most $\delta n^2$ edges. Despite its innocent appearance, this result has several striking consequences. Most notably, it easily implies Roth's theorem~\cite{R53} on $3$-term arithmetic progressions in dense subsets of the integers. For general graphs $H$, a similar statement holds~\cite{ADLRY94,EFR86, Fu95} (see also \cite{CF13, F11}): if an $n$-vertex graph contains $o(n^{v(H)})$ copies of $H$, then it may be made $H$-free by removing $o(n^2)$ edges. This result is known as the graph removal lemma. A sparse random version of the graph removal lemma was conjectured by \L uczak in \cite{Lu06} and proved, for strictly balanced $H$, by Conlon and Gowers~\cite{CG12}. Here, we apply our main result, Theorem~\ref{thm:main}, to extend this result to all graphs $H$.

\begin{theorem} \label{thm:removal-Gnp}
For every $\delta > 0$ and every graph $H$, there exist positive constants $\e$ and $C$ such that if $p \geq Cn^{-1/m_2(H)}$, then the following holds a.a.s.\ in $G_{n,p}$. Every subgraph of $G_{n,p}$ which contains at most $\e p^{e(H)} n^{v(H)}$ copies of $H$ may be made $H$-free by removing at most $\delta p n^2$ edges.
\end{theorem}

Note that if $p \le c n^{-1/m_2(H)}$  for a sufficiently small positive constant $c$ (depending on $H$ and $\delta$), then $H$ has a subgraph $H'$, the expected number of copies of which is at most $\delta pn^2$, so we can remove all copies of $H$ by deleting an edge from each copy of $H'$. Thus, it is natural to conjecture, as \L uczak did, that Theorem \ref{thm:removal-Gnp} actually holds for all values of $p$. For balanced graphs, we may close the gap by taking $\epsilon$ to be sufficiently small in terms of $C$, $\delta$, and $H$. For $p \leq C n^{-1/m_2(H)}$ and $\epsilon < \delta C^{-e(H)}$, the number of copies of $H$ is at most $\epsilon p^{e(H)} n^{v(H)} \leq \epsilon C^{e(H)} p n^2 < \delta p n^2$. Deleting an edge from each copy of $H$ yields the result.

\subsection{The clique density theorem} \label{sec:cliquedensity}

For any $\rho > 0$, let $g_k(\rho, n)$ be the minimum number of copies of $K_k$ which are contained in any graph on $n$ vertices with density at least $\rho$. We then take 
\[
g_k(\rho) = \lim_{n \rightarrow \infty} \frac{g_k(\rho, n)}{\binom{n}{k}}.
\] 
Complete $(k-1)$-partite graphs demonstrate that $g_k(\rho) = 0$ for $\rho \leq 1 - \frac{1}{k-1}$. On the other hand, a robust version of Tur\'an's theorem known as supersaturation \cite{ESi83} tells us that $g_k(\rho) > 0$ for all $\rho > 1 - \frac{1}{k-1}$.

There has been much work~\cite{B76, E62, E69, G59, LS83} on determining $g_k(\rho)$ above the natural threshold $1 - \frac{1}{k-1}$, but it is only in recent years that the exact dependency of $g_k(\rho)$ on $\rho$ has been found for any $k \geq 3$. For $k = 3$ and the interval $\frac{1}{2} \leq \rho \leq \frac{2}{3}$, this was accomplished by Fisher \cite{F89, GS00}, while for a general $\rho$ the value of $g_3(\rho)$ was determined by Razborov~\cite{R07, R08} using flag algebras. Employing different methods, Nikiforov~\cite{N11} reproved the $k = 3$ case and also solved the case $k = 4$. Finally, Reiher~\cite{Re12} recently resolved the general case.

We will show that an analogous theorem holds within the random graph $G_{n,p}$. We note that this theorem, Theorem~\ref{thm:sparsereiher} below, also follows as a direct application of the results of~\cite{CG12}. To state the result, we define, for a subgraph $G'$ of $G_{n,p}$, the relative density of $G'$ in $G_{n,p}$ to be $e(G')/p\binom{n}{2}$.

\begin{theorem} \label{thm:sparsereiher}
For any $k \geq 3$ and any $\e > 0$, there exists a constant $C > 0$ such that if $p \geq C n^{-2/(k+1)}$, then the following holds a.a.s.\ in $G_{n,p}$. Any subgraph $G'$ of $G_{n,p}$ will contain at least $(g_k(\rho) - \e) p^{\binom{k}{2}} \binom{n}{k}$ copies of $K_k$, where $\rho$ is the relative density of $G'$ in $G_{n,p}$.  
\end{theorem}

We remark that, as $m_2(K_k)=(k+1)/2$, the assumption on $p$ in the above theorem is best possible up to the value of the constant $C$. Using part (ii) of Theorem~\ref{thm:main}, we may show that a similar theorem also holds for any strictly balanced graph $H$. That is, if $g_H(\rho)$ is the natural analogue of $g_k(\rho)$ defined for $H$, then for any $\e > 0$, there exists a $C > 0$ such that if $p \geq C n^{-1/m_2(H)}$, the random graph $G_{n,p}$ will a.a.s.\ be such that any subgraph $G'$ of $G_{n,p}$ contains at least $(g_H(\rho) - \e) p^{e(H)} \frac{n^{v(H)}}{Aut(H)}$ copies of $H$, where again $\rho$ is the relative density of $G'$ in $G_{n,p}$. However, the function $g_H(\rho)$ is only well understood for cliques and certain classes of bipartite graph (see, for example, \cite{CFS10, LSz12}).

\subsection{The Hajnal-Szemer\'edi theorem}

Let $H$ be a fixed graph on $h$ vertices. An arbitrary collection of vertex-disjoint copies of $H$ in some larger graph is called an $H$-\emph{packing}. A \emph{perfect $H$-packing} (or $H$-\emph{factor}) is an $H$-packing that covers all vertices of the host graph. It has long been known for certain graphs $H$ that if the minimum degree of an $n$-vertex graph $G$ is sufficiently large and $n$ is divisible by $h$, then $G$ contains an $H$-factor. For example, Dirac's theorem~\cite{D52} implies that if $H$ is a path of length $h-1$, $n$ is divisible by $h$, and $\delta(G) \geq n/2$, then $G$ contains an $H$-factor. Corr{\'a}di and Hajnal~\cite{CH63} proved that $\delta(G) \geq 2n/3$ implies the existence of a $K_3$-factor in $G$. A milestone in this area of research, the famous theorem of Hajnal and Szemer{\'e}di~\cite{HS70}, states that the condition $\delta(G) \geq (1-\frac{1}{k})n$ is sufficient to guarantee a perfect $K_k$-packing in $G$ for an arbitrary $k$ (see \cite{KiKo08} for a short proof of this theorem).

\begin{theorem} \label{thm:hajnalszemeredi}
  For any $k \geq 3$, every $n$-vertex graph $G$ with $\delta(G) \geq (1-\frac{1}{k})n$ contains a $K_k$-factor, provided that $n$ is divisible by $k$.
\end{theorem}


This theorem has also been generalized to arbitrary $H$ \cite{AY96, KSS01, KO092}. In particular, a result of Koml\'os \cite{Ko00} shows that a parameter known as the critical chromatic number governs the existence of almost perfect packings (i.e., packings covering all but a $o(1)$-fraction of the vertices of the host graph) in graphs of large minimum degree. 

We will prove the following approximate version of the Hajnal-Szemer{\'e}di Theorem in the random graph $G_{n,p}$.

\begin{theorem} \label{thm:sparsehajnalszemeredi}
  For any $k \geq 3$ and any $\gamma > 0$, there exists a constant $C > 0$ such that if $p \geq Cn^{-2/(k+1)}$, then the following holds a.a.s.\ in $G_{n,p}$. Any subgraph $G'$ of $G_{n,p}$ with $\delta(G') \ge (1-\frac{1}{k}+\gamma)pn$ contains a $K_k$-packing that covers all but at most $\gamma n$ vertices.
\end{theorem}

We remark that the assumption on $p$ in the above theorem is best possible up to the value of the constant $C$. Using the result of Koml{\'o}s~\cite{Ko00}, we may show that a similar theorem also holds for any graph $H$. That is, for any $H$ and any $\gamma > 0$, there exists a $C > 0$ such that if $p \geq Cn^{-1/m_2(H)}$, then a.a.s.\ every subgraph $G'$ of $G_{n,p}$ satisfying $\delta(G') \geq (1-\frac{1}{\chi_{\text{cr}}(H)}+\gamma)pn$, where $\chi_{\text{cr}}(H)$ is the critical chromatic number of $H$, contains an $H$-packing covering all but at most $\gamma n$ vertices.

Finally, we remark that the problem of finding perfect packings in subgraphs of random graphs seems to be much more difficult. On the positive side, a best possible sparse random analogue of Dirac's theorem was recently proved by Lee and Sudakov~\cite{LS12}. On the negative side, it was observed by Huang, Lee, and Sudakov~\cite{HLS12} that for every $\e > 0$, there are $c, C > 0$ such that if every vertex of $H$ is contained in a triangle and $C n^{-1/2} \leq p \leq c$, then $G_{n,p}$ a.a.s.\ contains a spanning subgraph $G'$ with $\delta(G') \geq (1-\e)pn$ such that at least $\e p^{-2}/3$ vertices of $G'$ are not contained in a copy of~$H$. Therefore, the presence of the set of uncovered vertices in the statement of Theorem~\ref{thm:sparsehajnalszemeredi} is indispensable. For further discussion and related results, we refer the reader to~\cite{BLS12,HLS12}.

\subsection{The Andr\'asfai-Erd\H{o}s-S\'os theorem}

A result of Zarankiewicz \cite{Z47} states that if a graph on $n$ vertices has minimum degree at least $(1 - \frac{1}{k-1}) n$ then it contains a copy of $K_k$. This result follows immediately from Tur\'an's theorem but is interesting because it has a surprisingly robust stability version, due to Andr\'asfai, Erd\H{o}s, and S\'os \cite{AES74}. This theorem states that any $K_k$-free graph on $n$ vertices with minimum degree at least $(1 - \frac{3}{3k-4}) n$ must be $(k-1)$-partite.

This result was extended to general graphs by Alon and Sudakov \cite{AS06} (see also \cite{Al10}), who showed that for any graph $H$ and any $\gamma > 0$, every $H$-free graph on $n$ vertices with minimum degree at least $\left(1 - \frac{3}{3\chi(H)-4} + \gamma\right) n$ may be made $(\chi(H)-1)$-partite by deleting $o(n^2)$ edges. We prove a random analogue of this result, as follows.

\begin{theorem} \label{thm:aes}
For every graph $H$ and every $\gamma > 0$, there exists $C > 0$ such that if $p \geq C n^{-1/m_2(H)}$, then a.a.s.\ every $H$-free subgraph $G' \subseteq G_{n,p}$ with $\delta(G') \geq \left(1-\frac{3}{3\chi(H)-4} + \gamma\right)p n$ may be made $(\chi(H)-1)$-partite by removing from it at most $\gamma p n^2$ edges.
\end{theorem}

More generally, for any $H$ and any $r$, we may define
\[\delta_{\chi}(H, r) := \inf\{d \colon \delta(G) \geq d|G| \mbox{ and } H \not\subseteq G \Rightarrow \chi(G) \leq r\}.\]
The Andr\'asfai-Erd\H{o}s-S\'os theorem determines the value of $\delta_{\chi}(K_k, k-1)$, but there are also some results known for other values of $r$. For example, it is known \cite{BT12, H82, J95} that $\delta_{\chi}(K_3, 3) = \frac{10}{29}$ and $\delta_{\chi}(K_3,4) = \frac{1}{3}$. Our methods easily allow us to transfer any such results about $\delta_{\chi}(H,r)$ to the random setting. These are approximate results, proving that any $H$-free graph with a certain minimum degree is close to a graph with bounded chromatic number. As noted in \cite{ABGKM12}, one cannot hope to achieve a more exact result saying that the graph itself has bounded chromatic number.

\section{K\L R conjecture via multiple exposure}

\label{sec:proof-part-a}

In this section, we prove part~(i) of Theorem~\ref{thm:main}. The proof follows the ideas of~\cite{S12, Sj12}. The main idea in these proofs is to expose the random graph in multiple rounds. In this context, this can be traced back to the work of R\"odl and Ruci\'nski in~\cite{RR95}. 

We are going to prove a somewhat stronger statement, Theorem~\ref{thm:part-i} below, which easily implies part~(i) of Theorem~\ref{thm:main}. A bipartite graph between sets $U$ and $V$ is $(\e, d)$-lower-regular if, for every $U' \subseteq U$ and $V' \subseteq V$ with $|U'| \geq \e |U|$ and $|V'| \geq \e |V|$, the density $d(U',V')$ of edges between $U'$ and $V'$ satisfies $d(U',V') \geq d$. Given a graph $H$ on the vertex set $\{1, \ldots, k\}$, we denote by $\mathcal{G}_\ell(H, n,d,\e)$ the collection of all graphs $G$ on the vertex set $V_1 \cup \ldots \cup V_k$, where $V_1, \ldots, V_k$ are pairwise disjoint sets of size $n$ each, whose edge set consists of $e(H)$ different $(\e,d)$-lower-regular bipartite graphs, one graph between $V_i$ and $V_j$ for each $ij \in E(H)$. For an arbitrary graph $G$ and $p \in [0,1]$, we will denote by $G_p$ the random subgraph of $G$, where each edge of $G$ is included with probability $p$, independently of all other edges. Finally, given $H' \subseteq H$ and two graphs $G, G' \in \mathcal{G}_\ell(H, n,0,1)$, we denote by $\C(H,G;H',G')$ the set of all canonical copies of $H$ in $G$ such that the edges of $H'$ in these copies are in $G'$.

\begin{theorem}
  \label{thm:part-i}
  Let $H$ be an arbitrary graph. For every $H' \subseteq H$ and every $d > 0$, there exist $\e, \xi, b, C > 0$ and an integer $n_0$ such that if $n \geq n_0$ and $p \geq Cn^{-1/m_2(H)}$, then the following holds. For every $G \in \mathcal{G}_\ell(H,n,d,\e)$, with probability at least $1 - \exp(-bpn^2)$, the random graph $G_p$ has the following property: Every subgraph $G'$ of $G_p$ in $\mathcal{G}_\ell(H,n,dp,\e)$ satisfies 
\[|\C(H,G;H',G')| \geq \xi p^{e(H')}n^{v(H)}.\]
\end{theorem}

We now deduce the first part of Theorem~\ref{thm:main} from Theorem~\ref{thm:part-i}.

\begin{proof}[Proof of part~(i) of Theorem~\ref{thm:main}]
  Let $H$ be an arbitrary $k$-vertex graph and let $d > 0$. We may assume that $H$ contains a vertex of degree at least two (and hence $m_2(H) \geq 1$) since otherwise the assertion of the theorem is trivial. Let
  \[
  \e = \min\big\{ d/2, \e_{\ref{thm:part-i}}(H, d/2) \big\}, \quad b' = \min\big\{1/40, b_{\ref{thm:part-i}}(H,d/2)\big\}, \quad \text{and} \quad C' = C_{\ref{thm:part-i}}(H,d/2).
  \]
  Moreover, let $\xi' = \xi_{\ref{thm:part-i}}(H,d/2)$ and let $\xi = 2^{-e(H)}\xi'$. Next, fix an arbitrary positive constant $\eta$ and let $C = \max\{ C', 40 k \} \cdot \eta^{-2}$. Assume that $p \ge CN^{-1/m_2(H)}$ and that $N$ is sufficiently large. Finally, let $n \geq \eta N$ and let $m \geq dpn^2$. We estimate the probability of the event (we denote it by $\B$) that $G_{N,p}$ contains a subgraph $G \in \mathcal{G}(H,n,m,p,\e)$ with $G(H) < \xi \left(\frac{m}{n^2}\right)^{e(H)} n^{v(H)}$.

  Let $W_1, \ldots, W_k$ be pairwise disjoint subsets of $V(G_{N,p})$, each of size $n$, and let $\B(W_1, \ldots, W_k)$ denote the event that $G_{N,p}$ contains a subgraph $G$ as above with sets $W_1, \ldots, W_k$ playing the role of $V_1, \ldots, V_k$ from the definition of $\mathcal{G}(H,n,m,p,\e)$. By Chernoff's inequality (see below),
  \[
  \Pr\left(e_{G_{N,p}}(W_i, W_j) \geq 2pn^2 \text{ for some $ij \in E(H)$}\right) \leq e(H) \cdot \exp(-pn^2/16) \leq \exp(-pn^2/20).
  \]
  On the other hand, if $m \leq 2pn^2$, then by Theorem~\ref{thm:part-i} with $H = H'$, $d_{\ref{thm:part-i}} = d/2$ and $G_{\ref{thm:part-i}} = H(W_1, \ldots, W_k)$, the complete blow-up of $H$ obtained by replacing the vertices of $H$ by sets $W_1, \ldots, W_k$ and the edges by complete bipartite graphs, the probability that $G_{N,p} \cap H(W_1, \ldots, W_k)$ contains a subgraph $G \in \mathcal{G}(H,n,m,p,\e) \subseteq \mathcal{G}_{\ell}(H,n,pd/2,\e)$ satisfying
  \[
  G(H) < \xi \left(\frac{m}{n^2}\right)^{e(H)} n^{v(H)} \leq 2^{e(H)}\xi p^{e(H)} n^{v(H)} = \xi' p^{e(H)} n^{v(H)}
  \]
  is at most $\exp(-bpn^2)$. To see this, note that $H(W_1, \ldots, W_k)_p$ has the same distribution as $G_{N,p} \cap H(W_1, \ldots, W_k)$ and that our assumptions imply that $p \geq C N^{-1/m_2(H)} \geq C' n^{-1/m_2(H)}$. It follows that
  \[
  \Pr(\B) \leq \sum_{W_1, \ldots, W_k} \Pr\big(\B(W_1, \ldots, W_k)\big) \leq (k+1)^N \big[\exp(-pn^2/20) + \exp(-bpn^2)\big] \leq \exp(-b'\eta^2pN^2/2)
  \]
  since $pn^2 \geq \eta^2pN^2 \geq C\eta^2N^{2-1/m_2(H)} \geq 40kN$.
\end{proof}

Note that we used Chernoff's bound in the following standard form (see, for example,~\cite[Appendix~A]{AS08}).

\begin{lemma}[Chernoff's inequality]
Let $t$ be a positive integer, $p \in [0,1]$, and $X \sim Bin(t,p)$. For every positive $a$, 
\[\Pr(X < pt - a) < \exp\left(-\frac{a^2}{pt}\right) \quad \text{and} \quad \Pr(X > pt + a) < \exp\left(-\frac{a^2}{2pt} + \frac{a^3}{2(pt)^2} \right).\]
\end{lemma}

In particular, if we apply the upper tail estimate with $a = pt/2$, we see that 
\[\Pr(X > 2pt) \leq \Pr(X > 3pt/2) \leq \exp(-p t/16),\]
as required in the proof above with $t = n^2$.

In the proof of Theorem~\ref{thm:part-i}, we will also need the following approximate concentration result for random subgraphs of $H(n)$, the complete blow-up of $H$ obtained by replacing the vertices of $H$ by disjoint sets $V_1, \ldots, V_k$ of size $n$ each and the edges of $H$ by complete bipartite graphs. Lemma~\ref{lemma:part-i-upper-tail} below is \cite[Proposition~3.6]{S12} with $H_n$ being the $e(H)$-uniform hypergraph on the vertex set $E(H(n))$ whose edges are the canonical copies of $H$ in $H(n)$. The proof of the fact that this hypergraph is $(K,n^{-1/m_2(H)})$-bounded, see~\cite[Definition~3.2]{S12}, is implicit in the proof of the $1$-statement of \cite[Theorem~2.7]{S12}. The definition of $\deg_{H'}$ is given in Section~\ref{sec:part-i-estimating-P*S*}.

\begin{lemma}[{\cite{RR95,S12}}]
  \label{lemma:part-i-upper-tail}
  Let $H$ be an arbitrary graph with $\Delta(H) \ge 2$. There exists a $K > 0$ such that for every proper $H' \subseteq H$ and every $\eta > 0$, there exist $b > 0$ and an integer $n_0$ such that for~every $n \geq n_0$, if $p \geq n^{-1/m_2(H)}$, then with probability at least $1 - \exp(-bpn^2)$, for every $ij \in E(H) \setminus E(H')$ there exists a subgraph $X \subseteq H(n)_p$ with $|X| \leq \eta pn^2$ satisfying
  \[
  \sum_{e \in (V_i,V_j)} \deg_{H'}^2\left(e, H(n), H(n)_p \setminus X\right) \leq K p^{2e(H')}n^{2v(H)-2}.
 \]
\end{lemma}

Let $H$ be an arbitrary graph and let $d$ be a positive constant. In the remainder of this section, we prove Theorem~\ref{thm:part-i} by induction on the number of edges in $H'$.

\subsection{Induction base ($e(H') = 0$)}

\label{sec:proof-part-i-base}

Let $H'$ be the empty subgraph of $H$ and note that, regardless of $G'$, we have $|\C(H,G;H',G')| = G(H)$. The base of the induction follows immediately from the following one-sided version of the counting lemma, Lemma~\ref{lemma:H-count}, if we let $\e = \e_{\ref{lemma:H-count-lower}}(d)$, $\xi = \xi_{\ref{lemma:H-count-lower}}(d)$, $b = 1$, $C = 1$, and $n_0 = 1$.

\begin{lemma}
  \label{lemma:H-count-lower}
  For every graph $H$ and every $d > 0$, there exist $\e, \xi > 0$ such that for every $n$ and every $G \in \mathcal{G}_\ell(H,n,d,\e)$,
  \[
  G(H) \geq \xi n^{v(H)}.
  \]
\end{lemma}

\noindent
In fact, by choosing $\e$ sufficiently small, one can show that $\xi \geq d^{e(H)} - \delta$.

\subsection{Induction step}

\label{sec:proof-part-i-step}

Let $H''$ be an arbitrary subgraph of $H'$ with $e(H') - 1$ edges. Let $\e$, $\xi'$, $b'$, $C'$, and $n_0'$ be the constants whose existence is asserted by the inductive assumption with $H'$ replaced by $H''$ and $d$ replaced by $d/4$, i.e., let
\[
\e = \e_{\ref{thm:part-i}}(H'', d/4), \quad \xi' = \xi_{\ref{thm:part-i}}(H'', d/4), \quad b' = b_{\ref{thm:part-i}}(H'', d/4), \quad \text{and} \quad C' = C_{\ref{thm:part-i}}(H'', d/4).
\]
We also let $K = K_{\ref{lemma:part-i-upper-tail}}(H)$, $\eta = \e^2d/8$, and $\hat{b} = b_{\ref{lemma:part-i-upper-tail}}(H, H'', \eta)$. Furthermore, let
\begin{equation}
  \label{eq:R}
  R = \left\lceil \frac{4K}{(\xi')^2} + 1 \right\rceil
\end{equation}
and let
\[
\xi = \frac{\xi' \e^2 d}{16(RL^R)^{e(H)}},
\quad
b = \min\left\{\frac{b^*}{2RL^R}, \frac{e(H)d}{21RL^R} \right\},
\quad\text{and}\quad
C = \frac{32RL^RC'}{\e^2d},
\]
where
\[
b^* = \min \left\{ \frac{\e^2 d}{16^2}, \frac{b'}{4}, \frac{\hat{b}}{4} \right\}
\quad\text{and}\quad
L = \frac{3e(H)}{b^*}.
\]
Throughout the proof, we will assume that $n \geq n_0$, where $n_0$ is sufficiently large. In particular, we will assume that $n_0$ is larger than the values of $n$ which come from applying Theorem~\ref{thm:part-i} and Lemma~\ref{lemma:part-i-upper-tail} in context. Finally, assume that $p$ satisfies $p \geq Cn^{-1/m_2(H)}$ and fix some $G \in \mathcal{G}_\ell(H, n,d,\e)$.

\subsubsection{Multiple exposure trick}

\label{sec:mult-expos-trick}

Let $\S$ denote the event that the random graph $G_p$ possesses the postulated property:

\smallskip
\noindent
\begin{center}
  \begin{tabular}{rp{0.85\textwidth}}
    $\S$: & Every subgraph $G' \subseteq G_p$ such that $G' \in \mathcal{G}_{\ell}(H, n,dp,\e)$ satisfies 
    $|\C(H,G;H',G')| \geq \xi p^{e(H')}n^{v(H)}$.
  \end{tabular}
\end{center}

\smallskip
\noindent
Following Samotij~\cite{Sj12}, we will consider a richer probability space that is in a natural correspondence with the space $\cP(G)$ of all subgraphs of $G$ equipped with the obvious probability measure $\Pr$, i.e., the distribution of the random graph $G_p$. To this end, let $p_1, \ldots, p_R \in [0,1]$ be the unique numbers that satisfy
\begin{equation}
  \label{eq:part-i-ps}
  1 - p = \prod_{s=1}^R (1 - p_s) \quad\text{and}\quad p_{s+1} = L p_s \text{ for every $s \in [R-1]$}
\end{equation}
and observe that
\begin{equation}
  \label{eq:part-i-ps-bounds}
  \sum_{s = 1}^R p_s \geq p \quad\text{and consequently}\quad p_s \geq p_1 \geq \frac{p}{RL^R}\text{ for every $s \in [R]$}.
\end{equation}
The richer probability space will be the space $\cP(G)^R$ equipped with the product measure $\Pr^*$ that is the distribution of the sequence $(G_{p_1}, \ldots, G_{p_R})$ of independent random variables, where for each~$s$, the variable $G_{p_s}$ is a $p_s$-random subgraph of $G$. Crucially, observe that due to our choice of $p_1, \ldots, p_s$, see~\eqref{eq:part-i-ps}, the natural mapping
\[
\varphi \colon \cP(G)^R \to \cP(G) \quad \text{defined by} \quad \varphi(G_1, \ldots, G_R) = G_1 \cup \ldots \cup G_R
\]
is measure preserving, i.e., for every $G_0 \subseteq G$,
\[
\Pr(G_0) = \Pr^*(\varphi^{-1}(G_0)).
\]
In other words, the variables $G_p$ and $G_{p_1} \cup \ldots \cup G_{p_R}$ have the same distribution. Finally, let $d^* = d/2$ and consider the following event in the space $\cP(G)^R$:

\smallskip
\noindent
\begin{center}
  \begin{tabular}{rp{0.85\textwidth}}
    $\S^*$: & For every $G_1' \subseteq G_{p_1}, \ldots, G_R' \subseteq G_{p_R}$ such that $G_s' \in \mathcal{G}_{\ell}(H, n,d^*p_s,\e)$ for every $s \in [R]$, we have $|\C(H,G;H',G_1' \cup \ldots \cup G_R')| \geq \xi p^{e(H')}n^{v(H)}$.
  \end{tabular}
\end{center}

\smallskip
\noindent
There are two reasons why we consider the probability space $\cP(G)^R$. The first reason is that the probability of $\S^*$ is much easier to estimate than the probability of $\S$. The second reason is that a lower bound on $\Pr^*(\S^*)$ implies a (marginally weaker) lower bound on $\Pr(\S)$, which we show below.

\begin{claim}
  \label{claim:star}
  $1 - \Pr(\S) \leq 2 \cdot \big(1 - \Pr^*(\S^*)\big)$.
\end{claim}
\begin{proof}
  Note that in order to prove the claim, it suffices to show that
  \begin{equation}
    \label{eq:claim-star}
    \Pr^*(\S^* \mid \varphi^{-1}(\bG)) = \Pr^*(\S^* \mid G_{p_1} \cup \ldots \cup G_{p_R} = \bG) \leq 1/2
  \end{equation}
  for every $\bG$ that does \emph{not} satisfy $\S$. Indeed, assuming that~\eqref{eq:claim-star} holds for every such $\bG$, we have
  \[
  \begin{split}
    1 - \Pr(\S) & = \sum_{\bG \not\in \S} \Pr(G_p = \bG) = \sum_{\bG \not\in \S} \Pr^*(G_{p_1} \cup \ldots \cup G_{p_R} = \bG) \\
    & \le \sum_{\bG \not\in \S} 2\left(1 - \Pr^*(\S^* \mid G_{p_1} \cup \ldots \cup G_{p_R} = \bG)\right) \cdot \Pr^*(G_{p_1} \cup \ldots \cup G_{p_R} = \bG) \\
    & \le 2 \sum_{\bG} \left(1 - \Pr^*(\S^* \mid G_{p_1} \cup \ldots \cup G_{p_R} = \bG)\right) \cdot \Pr^*(G_{p_1} \cup \ldots \cup G_{p_R} = \bG) \\
    & = 2(1 - \Pr^*(\S^*)).
  \end{split}
  \]

  Consider an arbitrary $\bG \subseteq G$ that does not satisfy $\S$. By the definition of $\S$, there exists a subgraph $G' \subseteq \bG$ such that $G' \in \mathcal{G}_{\ell}(H,n,dp,\e)$ but $|\C(H,G;H',G')| < \xi p^{e(H')} n^{v(H)}$. Consider the event $\varphi^{-1}(\bG)$, i.e., the event $G_{p_1} \cup \ldots \cup G_{p_R} = \bG$. Now, for each $s \in [R]$, let $G_s' = G' \cap G_{p_s}$. Since clearly $\C(H,G;H',G') = \C(H,G;H',G_1' \cup \ldots\cup G_R')$, it suffices to show that with probability at least $1/2$, we have $G'_s \in \mathcal{G}_{\ell}(H,n,d^*p_s,\e)$ for every $s \in [R]$.

  To this end, observe that conditioned on the event $G_{p_1} \cup \ldots \cup G_{p_R} = \bG$, for each $s \in [R]$, the variable $G_{p_s}$ has the same distribution as $\bG_{p'_s}$, where $p'_s = p_s / p$ (although $G_{p_1}, \ldots, G_{p_R}$ are no longer independent). Since $G' \in \G_{\ell} (H,n, dp, \e)$ and for each $s \in [R]$, the graph $G_s'$ is simply a $p'_s$-random subgraph of $G'$, it follows from Chernoff's inequality and the union bound that for fixed $s \in [R]$, the probability that $G_s' \not\in \G_{\ell}(H,n,d^*p_s,\e) = \mathcal{G}_{\ell}(H,n,dp \cdot p'_s/2,\e)$ is at most
  \[
  e(H) \cdot 2^{2n} \cdot \exp(-\e^2dp_sn^2/16).
  \]
  Since $p_s \geq p/(RL^R) \geq C n^{-1/m_2(H)}/(R L^R) \ge 32n^{-1}/(\e^2d)$, the claimed estimate follows from the union bound, provided that $n$ is sufficiently large.
\end{proof}

\subsubsection{Estimating the probability of $\S^*$}

\label{sec:part-i-estimating-P*S*}

In the remainder of the proof, we will work in the space $\cP(G)^R$ and estimate the probability of the event $\S^*$, that is, $\Pr^*(\S^*)$. Let $G_{p_1}, \ldots, G_{p_R}$ be independent random subgraphs of $G$. Given $G_1' \subseteq G_{p_1}, \ldots, G_R' \subseteq G_{p_R}$, let
\[
G(s) = (G_{p_1}, \ldots, G_{p_s}) \quad\text{and}\quad G'(s) = (G_1', \ldots, G_s').
\]
Let $ij$ be the edge of $H'$ that is missing in $H''$ and let $V_i$ and $V_j$ be the subsets of $V(G)$ corresponding to the vertices $i$ and $j$, respectively. We consider the set $Z_s$ of `rich' edges of $G(V_i, V_j)$ that belong to many copies of $H$ from $\C(H,G;H'',G_s')$, which we define by
\[
Z_s = \left\{ e \in G(V_i, V_j) \colon \deg_{H''}(e, G, G_s') \geq \frac{\xi'}{2} p_s^{e(H'')} n^{v(H)-2} \right\},
\]
where
\[
\deg_{H''}(e, G, G') = \left| \left\{ J \in \C(H,G;H'',G') \colon e \in J \right\} \right|,
\]
and let
\[
Z(s) = Z_1 \cup \ldots \cup Z_s.
\]
Now comes the key step in the proof. We show that with very high probability, for every $s \in [R]$, regardless of $G(s-1)$ and $G'(s-1)$, either $\C(H,G;H',G_1' \cup \ldots \cup G_s')$ is large or the set $Z(s)$ of `rich' edges grows by more than $n^2/R$, that is, $|Z(s) \setminus Z(s-1)| \ge n^2/R$.

\begin{claim}
  \label{claim:part-i-main}
  For every $s \in [R]$ and every choice of $G'(s-1) \subseteq G(s-1) \in \cP(G)^{s-1}$, let $\S_{G'(s-1)}^*$ denote the event that $G_{p_s}$ has the following property: For every $G_s' \subseteq G_{p_s}'$ that satisfies $G_s' \in \G_{\ell}(H,n,d^*p_s,\e)$, either
  \begin{equation}
    \label{eq:part-i-done}
    \left|\C(H,G;H',G_1' \cup \ldots \cup G_s')\right| \geq \xi p^{e(H')} n^{v(H)}
  \end{equation}
  or
  \begin{equation}
    \label{eq:Z-grows}
    \left|Z(s) \setminus Z(s-1)\right| \geq \frac{(\xi')^2}{4K} n^2.
  \end{equation}
  Then for every $\bG \in \cP(G)^{s-1}$,
  \[
  \Pr^*(\S^*_{G'(s-1)} \mid G(s-1) = \bG) \geq 1 - \exp(-2b^*p_sn^2),
  \]
  where $\Pr^*(\S^*_{G'(0)} \mid G(0) = \bG) = \Pr^*(\S^*_{G'(0)})$.
\end{claim}

\subsubsection{Deducing Theorem~\ref{thm:part-i} from Claims~\ref{claim:star} and~\ref{claim:part-i-main}}

For every $s \in [R]$, let $\A_s$ denote the event that $e(G_{p_s}) \leq 2p_s e(G)$ and let $\A(s) = \A_1 \cap \ldots \cap \A_s$. Observe that by~\eqref{eq:part-i-ps},
\begin{equation}
  \label{eq:part-i-Us-}
  \sum_{t=1}^{s-1} 2p_t e(G) \leq 3p_{s-1}e(G) = \frac{3p_se(G)}{L} \leq \frac{3e(H)p_sn^2}{L}.
\end{equation}
By Chernoff's inequality, \eqref{eq:part-i-ps}, \eqref{eq:part-i-ps-bounds}, and the fact that $e(G) \geq e(H) d n^2$,
\begin{equation}
  \label{eq:part-i-not-A}
  \Pr^*(\neg \A(s-1)) \leq \sum_{t=1}^{s-1} \exp\left(-\frac{p_t e(G)}{16} \right) \leq s \cdot \exp\left(-\frac{p_1e(G)}{16}\right) \leq \exp\left(-\frac{e(H)dp_1n^2}{20} \right).
\end{equation}
Now, for every $s \in [R]$, let $\S_s^*$ denote the event that $\A(s-1)$ holds and $\S_{G'(s-1)}^*$ holds for all $G'(s-1) \subseteq G(s-1)$. Here, $G'(s-1) \subseteq G(s-1)$ means that inclusion holds coordinate-wise. Observe that if $\S_s^*$ holds for all $s \in [R]$, then $\S^*$ must hold since~\eqref{eq:Z-grows} in Claim~\ref{claim:part-i-main} can occur at most $R-1$ times, see~\eqref{eq:R}. Let
\[
\Gh = \left\{ \bG \in \cP(G)^{s-1} \colon |\bG_t| \leq 2p_te(G) \text{ for all $t \in [s-1]$}\right\}
\]
and note that
\begin{align}
  \label{eq:part-i-P*Ss*}
  \Pr^*(\neg \S_s^*) & \leq \Pr^*(\neg \A(s-1)) + \Pr^*(\neg \S_s^* \wedge \A(s-1)) \\
  \nonumber
  &  = \Pr^*(\neg \A(s-1)) + \sum_{\bG \in \Gh} \Pr^*(\neg \S_s^* \wedge G(s-1) = \bG).
\end{align}
Now, by Claim~\ref{claim:part-i-main}, for every $\bG \in \Gh$,
\begin{align}
  \label{eq:part-i-P*Ss*bU}
  \Pr^*(\neg \S_s^* \wedge G(s-1) = \bG) & = \Pr^*(\neg \S_s^* \mid G(s-1) = \bG) \cdot \Pr^*(G(s-1) = \bG) \\
  \nonumber
  & \leq \sum_{G'(s-1) \subseteq \bG} \Pr^*(\neg \S_{G'(s-1)}^* \mid G(s-1) = \bG) \cdot \Pr^*(G(s-1) = \bG) \\
  \nonumber
  & \leq 2^{\sum_{t=1}^{s-1}2p_te(G)} \cdot \exp\left(-2b^*p_sn^2\right) \cdot \Pr^*(G(s-1) = \bG).
\end{align}
Since clearly $\sum_{\bG \in \Gh} \Pr^*(G(s-1) = \bG) \leq 1$, it follows from~\eqref{eq:part-i-Us-}, \eqref{eq:part-i-not-A}, \eqref{eq:part-i-P*Ss*}, and \eqref{eq:part-i-P*Ss*bU} that
\begin{align}
  \label{eq:part-i-P*S*}
  \Pr^*(\neg \S^*) & \leq \sum_{s=1}^R \Pr^*(\neg \S_s^*) \leq R \exp\left( -\frac{e(H)dp_1n^2}{20} \right) + \sum_{s=1}^R 2^{\frac{3e(H)p_sn^2}{L}}\exp\left(-2b^*p_sn^2\right) \\
  \nonumber
  & \leq R\exp\left(-\frac{e(H)dp_1n^2}{20}\right) + R\exp\left(-b^*p_1n^2\right) \leq \frac{1}{2}\exp\left(-bpn^2\right).
\end{align}
Now, Theorem~\ref{thm:part-i} easily follows from \eqref{eq:part-i-P*S*} and Claim~\ref{claim:star}.

\subsubsection{Proof of Claim~\ref{claim:part-i-main}}

Let $s \in [R]$, condition on the event $G(s-1) = \bG$ for some $\bG \in \cP(G)^{s-1}$, and assume that $G'(s-1)$ is given. Note that this uniquely defines the graph $Z(s-1)$ of `rich' edges. Also, observe that it follows from the definition of $Z(s-1)$ and~\eqref{eq:part-i-ps-bounds} that
\begin{align}
  \label{eq:part-i-CHGs}
  |\C(H,G;H',G_1' \cup \ldots \cup G_s')| & \geq \sum_{e \in G_s' \cap Z(s-1)} \deg_{H''}(e, G, G_1' \cup \ldots \cup G_{s-1}') \\
  \nonumber
  & \geq e\big(G_s' \cap Z(s-1)\big) \cdot \frac{\xi'}{2}p_1^{e(H'')} n^{v(H)-2} \\
  \nonumber
  & \geq \frac{e\big(G_s' \cap Z(s-1)\big)}{p_s} \cdot \frac{\xi'}{2}\frac{p^{e(H')}}{(RL^R)^{e(H')}} n^{v(H)-2} \\
  \nonumber
  & \geq \frac{e\big(G_s' \cap Z(s-1)\big)}{(\e^2d/8) p_s n^2} \cdot \xi p^{e(H')} n^{v(H)},
\end{align}
hence it will be enough if we show that
\begin{equation}
  \label{eq:part-i-Gs-Zs-}
  e\left(G_s' \cap Z(s-1)\right) \geq \frac{\e^2d}{8} p_s n^2.
\end{equation}
Let $B = G(V_i, V_j)$. We consider two cases, depending on whether the bipartite graph $B \setminus Z(s-1)$ of `poor' edges is $(\e,d^*/2)$-lower-regular.

\medskip
\noindent
{\bf Case 1. $B \setminus Z(s-1)$ is not $(\e,d^*/2)$-lower-regular}

\smallskip
\noindent
In this case, there exist sets $X_i \subseteq V_i$ and $X_j \subseteq V_j$ with $|X_i|, |X_j| \geq \e n$ such that
\[
e_{B \setminus Z(s-1)}(X_i, X_j) < (d^*/2)|X_i||X_j|.
\]
By Chernoff's inequality, with probability at least $1 - \exp(-2b^*p_sn^2)$, the graph $B_{p_s}$ satisfies
\[
e_{B_{p_s} \setminus Z(s-1)}(X_i,X_j) \leq (3d^*/4)p_s|X_i||X_j|.
\]
Consequently, since for every $G_s' \subseteq G$ satisfying $G_s' \in \G_{\ell}(H,n,d^*p_s,\e)$, the graph $G_s' \cap B$ is $(\e,d^* p_s)$-lower-regular, then
\begin{align*}
  e\left(G_s' \cap Z(s-1)\right) & \geq e_{G_s' \cap Z(s-1)}(X_i, X_j) \geq e_{G_s'}(X_i, X_j) - e_{B_{p_s} \setminus Z(s-1)}(X_i,X_j) \\
  & \geq (d^*/4)p_s|X_i||X_j| \geq (d^*/4)\e^2p_sn^2 = (d/8)\e^2p_sn^2,
\end{align*}
which, by~\eqref{eq:part-i-CHGs} and~\eqref{eq:part-i-Gs-Zs-}, proves~\eqref{eq:part-i-done}.

\medskip
\noindent
{\bf Case 2. $B \setminus Z(s-1)$ is $(\e,d^*/2)$-lower-regular}

\smallskip
\noindent
In this case, we will apply the inductive assumption to the graph $G \setminus Z(s-1)$, which clearly is $(\e,d^*/2)$-lower regular as $Z(s-1) \subseteq B \subseteq G$. First, observe that if $e\big(G_s' \cap Z(s-1)\big) \geq (d/8)\e^2p_sn^2$, then this, by~\eqref{eq:part-i-CHGs} and~\eqref{eq:part-i-Gs-Zs-}, proves~\eqref{eq:part-i-done}, so from now on we may assume that the opposite inequality holds, i.e., that
\begin{equation}
  \label{eq:part-i-Gs-Zs-neg}
  e\big(G_s' \cap Z(s-1)\big) < \frac{d}{8} \e^2 p_s n^2.
\end{equation}
Let $\tG = G \setminus Z(s-1)$ and observe again that $\tG \in \G_{\ell}(H,n, d/4, \e)$. Note that by~\eqref{eq:part-i-ps-bounds} and our assumption on $p$, we have $p_s \geq \frac{p}{RL^R} \geq \frac{Cn^{-1/m_2(H)}}{RL^R} \geq C'n^{-1/m_2(H)}$ and hence by the inductive assumption applied to $\tG$, with probability at least $1 - \exp(-b'p_sn^2)$, every subgraph $\tG' \subseteq \tG_{p_s}$ in $\G_{\ell}(H,n,(d/4)p_s,\e)$ satisfies
\[
|\C(H,\tG;H'',\tG')| \geq \xi'p_s^{e(H'')}n^{v(H)}.
\]
Moreover, it follows from Lemma~\ref{lemma:part-i-upper-tail} that with probability at least $1 - \exp(-\hat{b}p_s|V|)$, there exists a subgraph $X \subseteq \tG_{p_s}$ satisfying 
\begin{equation}
  \label{eq:part-i-X}
  e(X) \leq \eta p_s n^2
\end{equation}
and
\begin{equation}
  \label{eq:U'-L2-norm}
  \sum_{e \in \tG(V_i,V_j)} \deg_{H''}^2(e, \tG, \tG_{p_s} \setminus X) \leq K p_s^{2e(H'')}n^{2v(H)-2}.
\end{equation}
Consider the graph $\tG' \subseteq \tG_{p_s}$ defined by $\tG' = G_s' \setminus (X \cup Z(s-1))$. We verify that $\tG' \in \G_{\ell}(H,n, (d/4)p_s, \e)$. It follows from~\eqref{eq:part-i-Gs-Zs-neg} and~\eqref{eq:part-i-X} that for every $i'j' \in E(H)$ and every pair of sets $W_{i'} \subseteq V_{i'}$ and $W_{j'} \subseteq V_{j'}$ with $|W_{i'}| \geq \e n$ and $|W_{j'}| \geq \e n$,
\begin{align*}
  e_{\tG'}(W_{i'},W_{j'}) & \geq e_{G_s'}(W_{i'},W_{j'}) - e\big(G_s' \cap Z(s-1)\big) - e(X) \\
  & \geq \left( d^* - \frac{d}{8} - \frac{\eta}{\e^2} \right) p_s|W_{i'}||W_{j'}| \geq (d/4) p_s |W_{i'}||W_{j'}|.
\end{align*}
From the inductive assumption (which holds for $\tG_{p_s}$ with probability at least $1 - \exp\left(-b'p_sn^2\right)$), we infer that
\begin{equation}
  \label{eq:U'-L1-norm}
  \sum_{e \in \tG(V_i,V_j)} \deg_{H''}(e, \tG, \tG') = |\C(H,\tG;H'',\tG')| \geq \xi' p_s^{e(H'')} n^{v(H)}.
\end{equation}
Let
\[
Z_s' = \left\{ e \in \tG(V_i,V_j) \colon \deg_{H''}(e,\tG,\tG') \geq \frac{\xi'}{2}p_s^{e(H'')} n^{v(H)-2} \right\}
\]
and note that, by definition, $Z_s' \subseteq Z_s$ and, by~\eqref{eq:U'-L1-norm},
\begin{align}
  \label{eq:Zs'-L1-norm}
  \sum_{e \in Z_s'} \deg_{H''}(e, \tG, \tG') & \geq \sum_{e \in \tG(V_i,V_j)} \deg_{H''}(e, \tG, \tG') - e\big(\tG(V_i, V_j) \setminus Z_s'\big) \cdot \frac{\xi'}{2}p_s^{e(H'')} n^{v(H)-2} \\
  \nonumber
  & \geq \xi' p_s^{e(H'')} n^{v(H)} - n^2 \cdot \frac{\xi'}{2} p_s^{e(H'')} n^{v(H)-2} = \frac{\xi'}{2} p_s^{e(H'')} n^{v(H)}.
\end{align}
It follows from~\eqref{eq:U'-L2-norm}, \eqref{eq:Zs'-L1-norm}, and the Cauchy-Schwarz inequality that
\begin{align*}
  Kp_s^{2e(H'')} n^{2v(H)-2} & \geq \sum_{e \in \tG(V_i,V_j)} \deg_{H''}^2(e, \tG, \tG_{p_s} \setminus X) \geq \sum_{e \in \tG(V_i,V_j)} \deg_{H''}^2(e, \tG, \tG') \\
  & \geq \frac{1}{|Z_s'|} \left( \sum_{e \in Z_s'} \deg_{H''}(e, \tG, \tG') \right)^2 \geq \frac{1}{|Z_s'|} \left( \frac{\xi' p_s^{e(H'')} n^{v(H)}}{2}\right)^2
\end{align*}
and consequently,
\[
|Z_s'| \geq \frac{(\xi')^2}{4K}n^2.
\]
Since $Z_s' \subseteq \tG = G \setminus Z(s-1)$, the graphs $Z_s'$ and $Z(s-1)$ are disjoint. Therefore, \eqref{eq:Z-grows} holds with probability at least
\[
1 - \exp(-b'p_sn^2) - \exp(-\hat{b}p_sn^2),
\]
which is at least $1 - \exp(-2b^*p_s n^2)$. This concludes the proof of the claim.

\section{K\L R conjecture via transference}

The method employed by Conlon and Gowers \cite{CG12} to determine probabilistic thresholds for combinatorial theorems hinges on proving a transference principle \cite{G10, GT08, RTTV08}. Suppose, for example, that one wishes to prove Tur\'an's theorem for triangles within the random graph $G_{n, p}$ for $p \geq C n^{-1/2}$. Then they show that asymptotically almost surely $G_{n,p}$ has the following property. Any subgraph $G$ of $G_{n,p}$ may be modelled by a subgraph $K$ of the complete graph on $n$ vertices in such a way that the proportion of edges and triangles in the dense graph is close to the proportion of edges and triangles in the sparse graph. That is, if the sparse graph $G$ contains $c_1 p n^2$ edges and $c_2 p^3 n^3$ triangles then the dense model graph $K$ should contain approximately $c_1 n^2$ edges and $c_2 n^3$ triangles.

To see that this implies Tur\'an's theorem in the sparse graph, suppose that $G$ is a subgraph of $G_{n,p}$ with $\left(\frac{1}{2} + \epsilon\right) p \binom{n}{2}$ edges. Then, provided the approximation is sufficiently good, the model graph $K$ will have at least $\left(\frac{1}{2} + \frac{\epsilon}{2}\right) \binom{n}{2}$ edges. A robust version of Tur\'an's theorem (as discussed in Section \ref{sec:cliquedensity}) then implies that $K$ contains at least $c n^3$ triangles for some $c > 0$. Provided again that the approximation is sufficiently good, this implies that the original graph $G$ contains at least $\frac{c}{2} p^3 n^3$ triangles. 

More generally, we have the following theorem (Corollary 9.7 in \cite{CG12}) from which many such threshold results for graphs may be derived in a similar fashion. For any graph $G$, we let the function $G \colon V(G)^2 \rightarrow \{0,1\}$ be the characteristic function of $G$, given by $G(x,y) = 1$ if $xy \in E(G)$ and $0$ otherwise. For any fixed graph $H$ on $k$ vertices, we also let
\[\mu_H(G) = N^{-k} \sum_{x_1, \dots, x_k} \prod_{ij \in E(H)} G(x_i, x_j)\]
be the normalized count of homomorphisms from $H$ to $G$.

\begin{theorem} \label{thm:transfer}
For any strictly balanced graph $H$ and any $\e > 0$, there exist positive constants $C$ and $\lambda$ such that if $C N^{-1/m_2(H)} \leq p \leq \lambda$ then the following holds a.a.s.\ in the random graph $G_{N,p}$. For every subgraph $G$ of $G_{N,p}$, there exists a subgraph $K$ of $K_N$ such that 
\[
|p^{-e(H)}\mu_H(G) - \mu_H(K)| \leq \epsilon
\]
and, for all pairs of disjoint vertex subsets $U_1, U_2$ of $V(K_N)$,
\[
|p^{-1} \sum_{x, y} G(x,y) - \sum_{x, y} K(x,y)|\leq \e N^2,
\]
where the sums are taken over all $x \in U_1$ and $y \in U_2$.
\end{theorem}

As stated, the result in \cite{CG12} only gives the bound $p^{-e(H)} \mu_H(G) \geq \mu_H(K) - \epsilon$. This is all that is necessary for the applications given in that (and this) paper. However, the bound $p^{-e(H)} \mu_H(G) \leq \mu_H(K) + \epsilon$ also follows from a more careful analysis.

To be more explicit, we need to say a little about the method used in \cite{CG12}. For any collection of functions $h_1, \dots, h_{e(H)}$, we consider the function
\[\mu_H(h_1, \dots, h_{e(H)}) = N^{-k} \sum_{x_1, \dots, x_k} \prod_{ij \in E(H)} h_{ij}(x_i, x_j).\]
Let $\gamma$ be the associated measure of the random graph $G_{N,p}$. We define this as being $\gamma(x,y) = p^{-1}$ if $xy$ is an edge of $G_{N,p}$ and $0$ otherwise. Note that a.a.s.\ the value of $\|\gamma\|_1$ is $1 + o(1)$. Suppose that $k$ and $g$ are functions on the edge set of $G_{N,p}$ with $0 \leq k \leq 1$ and $0 \leq g \leq \gamma$. For example, $g$ could be (and usually is) the characteristic function of a subgraph $G$ of $G_{N,p}$, with each edge weighted by a factor of $p^{-1}$. 

One of the main ideas of \cite{CG12} was to define a norm $\|.\|$ with the property that if $\|g_i - k_i\| = o(1)$ for all $1 \leq i \leq e(H)$ then $\mu_H(g_1, \dots, g_{e(H)}) \geq \mu_H(k_1, \dots, k_{e(H)}) - o(1)$. Note that if $G_i$ is a graph and $g_i$ is the weighted characteristic function $g_i = p^{-1} G_i$, then
\[p^{-e(H)} \mu_H(G_1, \dots, G_{e(H)}) = \mu_H(g_1, \dots, g_{e(H)}) \geq \mu_H(k_1, \dots, k_{e(H)}) - o(1),\]
that is, the functions $k_i$ serve as dense models for the $G_i$ for the purposes of one-sided counting. The main result of \cite{CG12} is that for any function $g_i$ with $0 \leq g_i \leq \gamma$ such a dense model function $k_i$  with $0 \leq k_i \leq 1$ exists. In particular, the constant function $1$ serves as an appropriate model for the function $\gamma$ corresponding to the random graph. 

Suppose that $k$ serves as a model for $g$. Then, by the triangle inequality,
\[\|(\gamma - g) - (1 - k)\| \leq \|\gamma - 1\| + \|g - k\| = o(1),\]
that is $1 - k$ serves as a model for $\gamma - g$. It follows that, for any $1 \leq i \leq e(H)$, 
\begin{equation} \label{eqn:twoside}
\mu_H(g, \dots, g, \gamma - g, \gamma, \dots, \gamma) \geq \mu_H(k, \dots, k, 1 - k, 1, \dots, 1) - o(1),
\end{equation}
where in the $\mu_H$ on the left-hand side the first $i - 1$ terms are $g$, the $i$th term is $\gamma - g$ and the remaining terms are $\gamma$. The same holds for the $\mu_H$ on the right-hand side with $g$ replaced by $k$ and $\gamma$ by $1$. Note that, since $\mu_H$ is additive in each variable,
\begin{align*}
\mu_H(g, \dots, g, g) & = \mu_H(g, \dots, g, \gamma) - \mu_H(g, \dots, g, \gamma - g)\\
& = \mu_H(g, \dots, g, \gamma, \gamma) - \mu_H(g, \dots, g, \gamma - g, \gamma) - \mu_H(g, \dots, g, \gamma - g)\\
& \vdots\\
& =  \mu_H(\gamma, \dots, \gamma, \gamma) - \mu_H(\gamma - g, \dots, \gamma, \gamma) - \cdots - \mu_H(g, \dots, g, \gamma - g).
\end{align*}
Therefore, by \eqref{eqn:twoside}, 
\begin{align*}
\mu_H(g, \dots, g, g) & \leq  \mu_H(1, \dots, 1, 1) - \mu_H(1 - k, \dots, 1, 1) - \cdots - \mu_H(k, \dots, k, 1 - k) + o(1)\\
& = \mu_H(k, \dots, k, k) + o(1).
\end{align*}
That is, $|\mu_H(g) - \mu_H(k)| = o(1)$. Here we used that $\mu_H(\gamma, \dots, \gamma, \gamma) = \mu_H(1, \dots, 1, 1) + o(1)$, which follows from standard tail estimates (see, for example, \cite{JOR04}).

To recover the statement of Theorem \ref{thm:transfer}, where we refer to graphs rather than functions, we let $g = p^{-1} G$. This yields a function $k$ with $0 \leq k \leq 1$ such that $|p^{-1} \mu_H(G) - \mu_H(k)| = o(1)$. If we now choose a graph $K$ randomly by picking each edge $xy$ independently with probability $k(x,y)$, we will a.a.s.\ produce a graph $K$ with $H$-count close to $k$ (see, for example, the proof of Corollary 9.7 in \cite{CG12}). Choosing such a graph, we have $|p^{-e(H)} \mu_H(G) - \mu_H(K)| = o(1)$, as required. 

Because we are dealing with canonical homomorphisms of a graph $H$ with $k$ vertices to a $k$-partite graph $G$ with vertex sets $V_1, \dots, V_k$, it would be quite useful to have another version of Theorem \ref{thm:transfer} which captures this situation. To this end, let $H^*(V_1, \dots, V_k)$ be the class of graphs on vertex set $V_1 \cup \dots \cup V_k$, where $V_1, \dots, V_k$ are disjoint sets, such that the only edges lie between sets $V_i$ and $V_j$ with $ij \in E(H)$. Then, for any $G \in H^*(V_1, \dots, V_k)$, we let 
\[\mu^*_H(G) = N^{-k} \sum_{x_1 \in V_1, \dots, x_k \in V_k} \prod_{ij \in E(H)} G(x_i, x_j)\]
be the normalized count of canonical homomorphisms from $H$ to $G$. The following theorem may be proved by a minor modification of the proof of Theorem \ref{thm:transfer}. 

\begin{theorem} \label{thm:canonicaltransfer}
For any strictly balanced graph $H$ on $k$ vertices and any $\e > 0$, there exist positive constants $C$ and $\lambda$ such that if $C N^{-1/m_2(H)} \leq p \leq \lambda$ then the following holds a.a.s.\ in the random graph $G_{N,p}$. For all disjoint vertex subsets $V_1, \dots, V_k$ and every subgraph $G$ of $G_{N,p}$ in $H^*(V_1, \dots, V_k)$, there exists a subgraph $K$ of $K_N$ in $H^*(V_1, \dots, V_k)$ such that 
\begin{equation} \label{eqn:Happrox}
|p^{-e(H)}\mu^*_H(G) - \mu^*_H(K)| \leq \epsilon
\end{equation}
and, for all pairs of disjoint vertex subsets $U_1, U_2$ of $V(K_N)$,
\begin{equation} \label{eqn:edgeapprox}
|p^{-1} \sum_{x, y} G(x,y) - \sum_{x, y} K(x,y)|\leq \e N^2,
\end{equation}
where the sums are taken over all $x \in U_1$ and $y \in U_2$.
\end{theorem}

In proving part (ii) of Theorem~\ref{thm:main}, we will use the following slight variant of the dense counting lemma, Lemma \ref{lemma:H-count}. We let $\G(H,n,m,p,\theta,\epsilon)$ be defined in exactly the same way as $\mathcal{G}(H,n,m,p,\epsilon)$, except we now allow the number of edges between each pair $V_i$ and $V_j$ to be $m \pm \theta p n^2$.

\begin{lemma}
\label{lemma:densecount}
For every graph $H$ and every $\delta > 0$, there exist $\theta, \e > 0$ and an integer $n_0$ such that for every $n \geq n_0$, every $m$, and every $G \in \mathcal{G}(H,n, m, 1, \theta, \e)$,
\begin{equation*}
G(H) = \left(\frac{m}{n^2}\right)^{e(H)} n^{v(H)} \pm \delta n^{v(H)}.
\end{equation*}
\end{lemma}

Part (ii) of Theorem~\ref{thm:main} is now a relatively easy corollary of Theorem \ref{thm:canonicaltransfer}. 

\vspace{2mm} \noindent
{\it Proof of part (ii) of Theorem \ref{thm:main}.}
Let $H$ be a fixed graph on $k$ vertices and $d, \delta$ fixed positive constants.  Choose $\theta$ and $\e_{\ref{lemma:densecount}}$ so that the conclusion of Lemma~\ref{lemma:densecount} holds with $\delta_{\ref{lemma:densecount}} = \frac{d^{e(H)} \delta}{4}$. We let $\e = \frac{\e_{\ref{lemma:densecount}}}{6}$ and, for a fixed positive constant $\eta$,  
\[\e_{\ref{thm:canonicaltransfer}} = \min\left(\frac{\eta^2 \e_{\ref{lemma:densecount}}^3}{3}, \eta^2 \theta, \frac{d^{e(H)}\delta}{4}\right)\]
and then choose $C$ and $\lambda$ so that the conclusion of Theorem \ref{thm:canonicaltransfer} holds with $\e_{\ref{thm:canonicaltransfer}}$. 

Suppose now that $C N^{-1/m_2(H)} \leq p \leq \lambda$ and $G_{N,p}$ satisfies the conclusion of Theorem \ref{thm:canonicaltransfer}. Let $G$ be a subgraph of $G_{N,p}$ from the set $\G(H,n,m,p, 2\e)$ with vertex sets $V_1, \dots, V_k$, each of size $n \geq \eta N$. Let $K$ be the dense graph given by Theorem \ref{thm:canonicaltransfer}. If $V_i' \subseteq V_i$ and $V'_j \subseteq V_j$ the triangle inequality tells us that $|d_K(V'_i, V'_j) - d_K(V_i, V_j)|$ is at most 
\[|d_K(V'_i, V'_j) - p^{-1} d_G(V'_i, V'_j)|  + p^{-1} |d_G(V'_i, V'_j)  - d_G(V_i, V_j)| + |p^{-1} d_G(V_i, V_j) - d_K(V_i, V_j)|.\]
Suppose that $|V_i'| \geq \e_{\ref{lemma:densecount}} |V_i|$ and $|V_j'| \geq \e_{\ref{lemma:densecount}} |V_j|$. Then, since $G$ is $(2\e, p)$-regular between $V_i$ and $V_j$ and $\e_{\ref{lemma:densecount}} \geq 2 \e$, the middle term is at most $2\e$. By \eqref{eqn:edgeapprox}, since $V'_i$ and $V'_j$ are disjoint sets of size at least $\e_{\ref{lemma:densecount}} \eta N$, the first and third terms are each at most $(\e_{\ref{lemma:densecount}} \eta)^{-2} \e_{\ref{thm:canonicaltransfer}}$. We therefore see that
\[|d_K(V'_i, V'_j) - d_K(V_i, V_j)| \leq 2 \e + 2 (\e_{\ref{lemma:densecount}} \eta)^{-2} \e_{\ref{thm:canonicaltransfer}} \leq \e_{\ref{lemma:densecount}}.\]
Therefore, $K$ is $\e_{\ref{lemma:densecount}}$-regular. Note also that the number of edges between $V_i$ and $V_j$ in $K$ is 
\[p^{-1} m \pm \e_{\ref{thm:canonicaltransfer}} N^2 = p^{-1} m \pm \theta n^2,\]
since $\e_{\ref{thm:canonicaltransfer}} \leq \eta^2 \theta$. Applying Lemma \ref{lemma:densecount}, we see that for $n$ sufficiently large
\[K(H) = \left(\frac{p^{-1} m}{n^2}\right)^{e(H)} n^{v(H)} \pm \delta_{\ref{lemma:densecount}} n^{v(H)} = \left(1 \pm \frac{\delta}{4}\right) \left(\frac{p^{-1} m}{n^2}\right)^{e(H)} n^{v(H)},\]
where we used that $m \geq d p n^2$. Therefore, by \eqref{eqn:Happrox}, 
\begin{eqnarray} \label{eqn:G-bound}
\nonumber
G(H) = p^{e(H)} K(H) \pm \e_{\ref{thm:canonicaltransfer}} p^{e(H)} n^{v(H)} & = &  \left(1 \pm \frac{\delta}{4}\right) \left(\frac{m}{n^2}\right)^{e(H)} \pm \e_{\ref{thm:canonicaltransfer}}  p^{e(H)} n^{v(H)}\\
& = & \left(1 \pm \frac{\delta}{2}\right) \left(\frac{m}{n^2}\right)^{e(H)} n^{v(H)},
\end{eqnarray}
where we used that $\e_{\ref{thm:canonicaltransfer}} \leq \frac{d^{e(H)}\delta}{4}$. 

Suppose now that $p > \lambda$ and fix a $G \subseteq G_{N,p}$ in $\G(H,n,m,p,\e)$ with $n \geq \eta N$ and $m \geq d p n^2$. Form a random subgraph $G'$ of $G$ by choosing $m' = \lambda m$ edges in each pair $G(V_i, V_j)$ uniformly at random. Clearly,
  \begin{equation}
    \label{eq:EE-G'H2}
    \EE[G'(H)] = \lambda^{e(H)} G(H).
  \end{equation}

  A standard application of Hoeffding's inequality (see~\cite[Lemma~4.3]{GS05}) proves that with probability at least 
$1 - \exp(-cm') \geq 1 - \delta/2$, the graph $G'$ is in $\G(H,n,m',p',2\e)$, where $p' = \lambda p$, and hence, by~\eqref{eqn:G-bound} and~\eqref{eq:EE-G'H2},
  \[
  G(H) \geq \lambda^{-e(H)} \left(1-\frac{\delta}{2}\right)^2 \left(\frac{m'}{n^2}\right)^{e(H)} n^{v(H)} \geq (1-\delta)\left(\frac{m}{n^2}\right)^{e(H)}n^{v(H)}.
  \]
 We may assume, since it happens a.a.s., that the total number of copies of $H$ in $G_{N,p}$ does not exceed $2p^{e(H)}N^{v(H)}$ (see, for example, \cite{JOR04}). Therefore,
 \begin{align*}
  G(H) \leq \lambda^{-e(H)} \cdot [\Pr\bigl( G' \in \G(H,n,m',p',2\e) \bigr) \cdot  & \left(1+ \frac{\delta}{2}\right)\left(\frac{m'}{n^2}\right)^{e(H)}n^{v(H)}\\
  & + \Pr\bigl( G' \not\in \G(H,n,m',p',2\e) \bigr) \cdot 2p^{e(H)}N^{v(H)}].
  \end{align*}
  Since $n \geq \eta N$, $m \geq dpn^2$ and $\Pr(G' \not\in \G(H,n,m',p',2\e)) \leq \exp(-cm') \leq \delta (\lambda d)^{e(H)} \eta^{v(H)}/4$ if $n$ is sufficiently large, it follows that
  \[
  G(H) \leq (1+\delta)\left( \frac{m}{n^2} \right)^{e(H)} n^{v(H)}.
  \]
The result follows. \qed

\vspace{2mm}
We note that the second case, where $p > \lambda$, also follows as an immediate corollary of the counting lemma for pseudorandom graphs proved in \cite{CFZ12}. Indeed, this result is already strong enough to imply a counting lemma down to densities of about $N^{-c/\Delta(H)}$, where $\Delta(H)$ is the maximum degree of $H$. In addition, if one is only interested in one-sided counting, that is, in showing that $G(H) \geq (1 - \delta) (m/n^2)^{e(H)} n^{v(H)}$, then the results of \cite{CFZ12} apply for $p \geq N^{-c/d(H)}$, where $d(H)$ is the degeneracy of $H$.

\section{Applications}

\subsection{Preliminaries}

After applying the sparse regularity lemma, it is usually helpful to clean up the the regular partition, removing all edges which are not contained in a dense regular pair. The following standard lemma, which incorporates both the regularity lemma and this cleaning process, is sufficient for our purposes. Recall that a graph is $(\eta, p, D)$-upper-uniform if for all disjoint subsets $U_1$ and $U_2$ with $|U_1|, |U_2| \geq \eta |V(G)|$, the density of edges between $U_1$ and $U_2$ satisfies $d(U_1, U_2) \leq D p$. Here, for the sake of clarity of presentation, we make the additional assumption in this definition that if $U_1 = U_2 = U$, then $e_G(U) \leq D p \binom{|U|}{2}$.

\begin{proposition} \label{prop:sparseregclean}
  For every $\e, D > 0$ and every positive integer $t_0$, there exist $\eta > 0$ and a positive integer $T$ such that, for every $d > 0$, every graph $G$ with at least $t_0$ vertices which is $(\eta, p, D)$-upper-uniform contains a subgraph $G'$ with
  \[
  e(G') \geq e(G) - \left( \frac{D}{t_0} + 2D\e + d\right) \frac{pn^2}{2}
  \]
  that admits an equipartition $V_1, \dots, V_t$ of its vertex set into $t_0 \leq t \leq T$ pieces such that the following conditions hold.
  \begin{enumerate}
  \item
    There are no edges of $G'$ within $V_i$ for any $1 \leq i \leq t$.
  \item
    Every non-empty graph $G'(V_i,V_j)$ is $(\e,p)$-regular and has at least $dp|V_i||V_j|$ edges.
  \end{enumerate}
\end{proposition}
\begin{proof}
  Fix $\e$, $D$, and $t_0$ as in the statement of the proposition and let $T = T_{\ref{thm:sparsereg}}(\e, D, t_0)$ and $\eta = \min\{\eta_{\ref{thm:sparsereg}}(\e, D, t_0), \frac{1}{2T}\}$. Let $d > 0$, fix a $G$ as above, and apply the sparse regularity lemma, Theorem~\ref{thm:sparsereg}, to obtain an $(\e,p)$-regular partition $V_1, \ldots, V_t$ of the vertices of $G$ into $t_0 \leq t \leq T$ pieces. Let us delete from $G$ all edges that are contained in:
\begin{itemize}
\item
  one of the sets $V_1, \ldots, V_t$,
\item
  one of the at most $\e t^2$ pairs $(V_i, V_j)$ that are not $(\e, p)$-regular, or
\item
  one of the pairs $(V_i, V_j)$ that have fewer than $d p |V_i| |V_j|$ edges.
\end{itemize}
Denote the resulting graph by $G'$. Since $G$ is $(\eta, p, D)$-upper-uniform, we have $e_G(V_i) \leq Dp\frac{1}{2}(\frac{n}{t})^2$ and $e_G(V_i,V_j) \leq Dp(\frac{n}{t})^2$ for all $i$ and $j$. It follows that
\[
e(G) - e(G') \leq t \cdot Dp\frac{1}{2}\left(\frac{n}{t}\right)^2 + \e t^2 \cdot Dp\left(\frac{n}{t}\right)^2 + \binom{t}{2} \cdot d p \left(\frac{n}{t}\right)^2 \leq \left( \frac{D}{t_0} + 2D\e + d \right) \frac{pn^2}{2}.
\qedhere
\]
\end{proof}

In the proofs of our applications, we will need a version of the main theorem which allows us to have different densities between different pairs of vertex sets. To this end, given a graph $H$ on the vertex set $\{1, \ldots, k\}$ and a sequence $\bm = (m_{ij})_{ij \in E(H)}$ of integers, we denote by $\G(H,n, \bm, p, \e)$ the collection of all graphs $G$ obtained in the following way. The vertex set of $G$ is a disjoint union $V_1 \cup \ldots \cup V_k$ of sets of size $n$.  For each edge $ij \in E(H)$, we add to $G$ an $(\e,p)$-regular bipartite graph with $m_{ij}$ edges between the pair $(V_i, V_j)$. These are the only edges of $G$. As before, for any $G \in \G(H, n, \bm, p, \e)$, let us denote by $G(H)$ the number of canonical copies of $H$ in $G$.

\begin{proposition}
  \label{prop:main}
  For every graph $H$ and every $\delta, d > 0$, there exist $\e, \xi > 0$ with the following property. For every $\eta > 0$, there is a $C > 0$ such that if $p \geq C N^{-1/m_2(H)}$ then a.a.s.\ the following holds in $G_{N,p}$:
  \begin{enumerate}
  \item[(i)]
    For every $n \geq \eta N$, $\bm$ with $m_{ij} \geq d p n^2$ for all $ij \in E(H)$ and every subgraph $G$ of $G_{N,p}$ in $\G(H,n,\bm,p, \e)$,
    \begin{equation} 
      G(H) \geq \xi \left( \prod_{ij \in E(H)}\frac{m_{ij}}{n^2} \right) n^{v(H)}.
    \end{equation}
    
  \item[(ii)]
    Moreover, if $H$ is strictly balanced, that is, if $m_2(H) > m_2(H')$ for every proper subgraph $H'$ of $H$, then 
    \begin{equation} \label{eqn:exactcountprime}
      G(H) = (1 \pm \delta) \left( \prod_{ij \in E(H)}\frac{m_{ij}}{n^2} \right) n^{v(H)}.
    \end{equation}
        
  \end{enumerate}
\end{proposition}
\begin{proof}
  Fix $H$, $\delta$, and $d$ as in the statement of the proposition. We may assume that $\Delta(H) \ge 2$ (and hence $m_2(H) \ge 1$) as otherwise the assertion of the proposition is trivial.  Let $\e = \e_{\ref{thm:main}}(H, \delta/2, d)/2$ and $\xi = \xi_{\ref{thm:main}}(H, \delta/2, d)/2$. Moreover, fix some $\eta > 0$, let $C$ be a sufficiently large positive constant, and suppose that $p \geq CN^{-1/m_2(H)}$. First, we will show that if $G_{N,p}$ satisfies part~(i) of Theorem~\ref{thm:main}, which happens a.a.s., then it also satisfies part~(i) of Proposition~\ref{prop:main}. To this end, let $n \geq \eta N$, let $\bm$ satisfy $m_{ij} \geq dpn^2$ for all $ij \in E(H)$, and fix a $G \subseteq G_{N,p}$ in $\G(H,n,\bm,p,\e)$. Form a random subgraph $G'$ of $G$ by choosing $m = dpn^2$ edges in each pair $G(V_i, V_j)$ uniformly at random. Clearly,
  \begin{equation}
    \label{eq:EE-G'H}
    \EE[G'(H)] = \left( \prod_{ij \in E(H)} \frac{m}{m_{ij}} \right) \cdot G(H).
  \end{equation}
  A standard application of Hoeffding's inequality (see~\cite[Lemma~4.3]{GS05}) proves that with probability at least $1 - \exp(-cm)$, where $c > 0$ is an absolute constant, the graph $G'$ is in $\G(H,n,m,p,2\e)$. Hence, if $G_{N,p}$ satisfies part~(i) of Theorem~\ref{thm:main}, then with probability at least $1/2$,
  \[
  G'(H) \geq 2\xi \left(\frac{m}{n}^2\right)^{e(H)} n^{v(H)}
  \]
  and therefore~\eqref{eq:EE-G'H} implies that
  \[
  G(H) = \left( \prod_{ij \in E(H)} \frac{m_{ij}}{m} \right) \cdot \EE[G'(H)] \geq \left( \prod_{ij \in E(H)} \frac{m_{ij}}{m} \right) \cdot \xi \left(\frac{m}{n^2}\right)^{e(H)} n^{v(H)},
  \]
  as claimed.

  Now, we show that if $G_{N,p}$ satisfies part~(ii) of Theorem~\ref{thm:main} and the total number of copies of $H$ in $G_{N,p}$ does not exceed $2p^{e(H)}N^{v(H)}$, which happens a.a.s.\ (see, for example, \cite{JOR04}), then part~(ii) of this proposition is also satisfied. To this end, fix $\bm$ and $G$ as above, recall the definition of $G'$, and observe that since with probability at least $1 - \delta/2$, the graph $G'$ is in $\G(H,n,m,p,2\e)$, then by~\eqref{eqn:exactcount} and~\eqref{eq:EE-G'H},
  \[
  G(H) \geq \left( \prod_{ij \in E(H)} \frac{m_{ij}}{m} \right) \cdot \left(1-\frac{\delta}{2}\right)^2 \left(\frac{m}{n^2}\right)^{e(H)} n^{v(H)} \geq (1-\delta)\left(\prod_{ij \in E(H)} \frac{m_{ij}}{n^2}\right)n^{v(H)}.
  \]
  On the other hand, we also have
 \begin{align*}
  G(H) \leq \left(\prod_{ij \in E(H)} \frac{m_{ij}}{m} \right) \cdot [\Pr\bigl( G' \in \G(H,n,m,p,2\e) \bigr) \cdot  & \left(1+ \frac{\delta}{2}\right)\left(\frac{m}{n^2}\right)^{e(H)}n^{v(H)}\\
  & + \Pr\bigl( G' \not\in \G(H,n,m,p,2\e) \bigr) \cdot 2p^{e(H)}N^{v(H)}].
  \end{align*}
  Since $n \geq \eta N$, $m \geq dpn^2$ and $\Pr(G' \not\in \G(H,n,m,p,2\e)) \leq \exp(-cm) \leq \delta d^{e(H)} \eta^{v(H)}/4$ if $n$ is sufficiently large, it follows that
  \[
  G(H) \leq (1+\delta)\left(\prod_{ij \in E(H)} \frac{m_{ij}}{n^2} \right) n^{v(H)}.
  \qedhere
  \]
\end{proof}

\subsection{The sparse removal lemma}

Let $\delta > 0$ and let $H$ be an arbitrary (not necessarily balanced) graph. The proof of Theorem~\ref{thm:removal-Gnp} is a classical application of the regularity method. We start by defining a range of constants. For the sake of brevity, we let $k = v(H)$. Furthermore, let $d = \delta/2$, $t_0 = 2/d$, $D = 2$, $\e' = \min\{\e_{\ref{prop:main}}(H,d/2), \delta/8\}$, $\xi = \xi_{\ref{prop:main}}(H,d/2)$,  $T = T_{\ref{prop:sparseregclean}}(\e'/k, D, t_0)$, and $\eta = \min\{\eta_{\ref{prop:sparseregclean}}(\e'/k, D, t_0), 1/(kT)\}$. Finally, let $\e = \xi (d/2)^{e(H)} (kT)^{-v(H)}$ and $C = C_{\ref{prop:main}}(H, d/2, \eta)$. Suppose that $p \geq Cn^{-1/m_2(H)}$. We will show that the sparse $H$-removal lemma holds in $G_{n,p}$ a.a.s., that is, that every subgraph of $G_{n,p}$ with fewer than $\e p^{e(H)} n^{v(H)}$ copies of $H$ can be made $H$-free by removing from it at most $\delta p n^2$ edges. It clearly suffices to show that part~(i) of Proposition~\ref{prop:main} and $(\eta, p, D)$-upper-uniformity, which holds in $G_{n,p}$ a.a.s., imply the above property. To this end, assume that $G_{n,p}$ is $(\eta, p, D)$-upper-uniform and that part~(i) of Proposition~\ref{prop:main} holds in $G_{n,p}$. Let $G \subseteq G_{n,p}$ be a subgraph with fewer than $\e p^{e(H)} n^{v(H)}$ copies of $H$ and let $G'$ be a subgraph of $G$ satisfying the assertion of Proposition~\ref{prop:sparseregclean} with $\e_{\ref{prop:sparseregclean}} = \e'/k$. By our choice of parameters,
\[
e(G) - e(G') \leq \left(\frac{D}{t_0} + \frac{2D\e'}{k} + d\right) \frac{pn^2}{2} \leq \left(2d + \frac{\delta}{2}\right)\frac{pn^2}{2} < \delta p n^2.
\]
We claim that $G'$ is $H$-free. Since all edges of $G'$ lie in $(\frac{\e'}{k},p)$-regular pairs with edge density at least $dp$, if $G'$ contained a copy of $H$, there would be a graph $H'$ with $k' \leq k$ vertices which is a homomorphic image of $H$, pairwise disjoint sets $V_1', \ldots, V_{k'}'$ and a sequence $\bm' = (m_{ij}')_{ij \in E(H')}$ with $m_{ij}' \geq dp(\frac{n}{t})^2$ such that $G'[V_1' \cup \ldots \cup V_{v(H')}'] \in \G(H', \frac{n}{t}, \bm, p, \e'/k)$. Consequently, there would be pairwise disjoint sets $V_1, \ldots, V_k$ and a sequence $\bm = (m_{ij})_{ij \in E(H)}$ with $m_{ij} \ge \frac{dp}{2} \left( \frac{n}{kt} \right)^2$ such that $G'[V_1 \cup \ldots \cup V_k] \in \G(H, \frac{n}{kt}, \bm, p, \e')$. One can obtain such $V_1, \ldots, V_k$ by arbitrarily dividing each $V_i'$ into $k$ parts and choosing $k$ of these parts according to the homomorphism from $H$ to $H'$. It would follow that the number of copies of $H$ in $G'$, and therefore also in $G$, would exceed
\[
\xi \left(\prod_{ij \in E(H)} \frac{m_{ij}}{(n/(kt))^2}\right)\left(\frac{n}{kt}\right)^{v(H)} \geq \xi\left(\frac{dp}{2}\right)^{e(H)}\left(\frac{n}{kt}\right)^{v(H)} \geq \e p^{e(H)} n^{v(H)},
\]
a contradiction.

\subsection{The clique density theorem}

To begin, we note that if $W$ is a weighted graph on $n$ vertices with $0 \leq W \leq 1$ for which $\sum_{x, y} W(x,y) \geq \rho\binom{n}{2}$, then, for any $\theta$ and $n$ sufficiently large depending on $\theta$, 
\[\sum_{x_1, \dots, x_k} \prod_{1 \leq a < b \leq k} W(x_a,x_b) \geq (g_k(\rho) - \theta) n^k.\]
This follows from choosing a random graph $G$, picking each edge $xy$ independently with probability $W(x,y)$. The resulting graph will, with high probability, have a similar count of edges and $K_k$s to the weighted graph $W$ (see, for example, Corollary 9.7 in \cite{CG12}). The result then follows by applying the clique density theorem to the graph $G$ (and using the fact that $g_k(\rho)$ is uniformly continuous).

Let $k \geq 3$ and $\e > 0$. We start by defining constants. We choose $t_1$ such that, for $t \geq t_1$ and all $\rho$, any weighted graph $W$ on $t$ vertices with $\sum_{x, y} W(x,y) \geq \rho\binom{t}{2}$ also satisfies 
\[\sum_{i_1, \dots, i_k} \prod_{1 \leq a < b \leq k} W(i_a,i_b) \geq \left(g_k(\rho) - \frac{\e}{4}\right) t^k.\] 
Since $g_k(\rho)$ is uniformly continuous in $\rho$, we may choose $\delta'$ such that $|g_k(\rho \pm \delta') - g_k(\rho)| \leq \frac{\e}{4}$. Let $\delta = \min\{\delta', \frac{\e}{2}\}$, $d = \delta/16$, $t_0 = \max\{t_1, 2/d\}$, $D = 2$, $\e' = \min\{\e_{\ref{prop:main}}(H, \delta, d), \delta/32\}$, $T = T_{\ref{prop:sparseregclean}}(\e',D, t_0)$, and $\eta = \min\{\eta_{\ref{prop:sparseregclean}}(\e',D, t_0), 1/T\}$. Finally, let $C = C_{\ref{prop:main}}(H, \delta, d,\eta)$. 

Suppose that $p \geq Cn^{-2/(k+1)}$. Since $(\eta, p, D)$-upper-uniformity holds a.a.s.\ in $G_{n,p}$, we will assume that it is satisfied. Let $G' \subseteq G_{n,p}$ be a subgraph of $G_{n,p}$ of relative density $\rho$, that is, with $\rho p \binom{n}{2}$ edges. By Proposition~\ref{prop:sparseregclean} with $\e_{\ref{prop:sparseregclean}} = \e'$, we get an equipartition $V_1, \dots, V_t$ of the vertex set of $G'$ into $t_0 \leq t \leq T$ pieces and a subgraph $G''$ of $G'$ all of whose edges lie in $(\e',p)$-regular pairs with edge density at least $dp$ and which satisfies
\[
e(G') - e(G'') \leq \left(\frac{D}{t_0} + 2D\e' + d\right) \frac{pn^2}{2} \leq \left(2d + \frac{\delta}{8}\right)\frac{pn^2}{2} < \frac{\delta}{4} p n^2.
\]

Consider the reduced weighted graph $R$ on vertex set $[t]$ with the weight of edge $ij$ given by 
$$R(i,j) = \min\left(\frac{e_{G''}(V_i, V_j)}{p|V_i||V_j|}, 1\right).$$ 
Note that a.a.s.\ the random graph $G_{n,p}$ satisfies $e_{G}(U,V) \leq p |U||V| + \frac{\delta}{4 T^2} p n^2$ for all disjoint subsets $U$ and $V$. This in turn implies that $e_{G''}(V_i,V_j) \leq p |V_i||V_j| + \frac{\delta}{4 T^2} p n^2$. Hence,
$$R(i,j)  \geq \frac{e_{G''}(V_i, V_j)- \frac{\delta}{4} p \left(\frac{n}{T}\right)^2}{p|V_i||V_j|}.$$ 
Therefore, since $e(G') \geq \rho p \binom{n}{2}$,
\[\sum_{1 \leq i < j \leq t} R(i, j) \geq \frac{e(G'') - \frac{\delta}{4} p n^2}{p(\frac{n}{t})^2} \geq \frac{e(G') - \frac{\delta}{2} p n^2}{p (\frac{n}{t})^2} \geq (\rho - \delta) \binom{t}{2}.\]
Therefore, by the choice of $t_1$ and $\delta'$, we have that 
\[\sum_{i_1, \dots, i_k} \prod_{1 \leq a < b \leq k} R(i_a, i_b) \geq \left(g_k(\rho - \delta) - \frac{\epsilon}{4}\right) t^k \geq \left(g_k(\rho) - \frac{\epsilon}{2}\right) t^k.\]
Now, for any particular $1 \leq i_1 < \dots <  i_k \leq t$, consider the sets $V_{i_1}, \dots, V_{i_k}$. Then $G''[V_{i_1} \cup \dots \cup V_{i_k}]$ is an element of $\G(K_k,\frac{n}{t}, \bm, p, \e')$ with $m_{i_a i_b} \ge R(i_a, i_b) p \left(\frac{n}{t}\right)^2$. By part (ii) of Proposition~\ref{prop:main} and the choice of $\e'$, $\eta$, and $C$, it follows that the number of copies of $K_k$ between the sets $V_{i_1}, \dots, V_{i_k}$ is at least
\[(1 - \delta) \prod_{1 \leq a < b \leq k} R(i_a, i_b) \cdot p^{\binom{k}{2}} \left(\frac{n}{t}\right)^{k}.\]
Adding over all choices of $i_1, \dots, i_k$ gives
\begin{align*}
\mu_{G''}(K_k) & \geq (1 - \delta) \sum_{i_1, \dots, i_k} \prod_{1 \leq a < b \leq k} R(i_a, i_b) \cdot p^{\binom{k}{2}} \left(\frac{n}{t}\right)^{k}\\ 
& \geq \left(1 -\frac{\e}{2}\right) \left(g_k(\rho) - \frac{\epsilon}{2}\right) t^k p^{\binom{k}{2}} \left(\frac{n}{t}\right)^{k} \geq (g_k(\rho) - \epsilon) p^{\binom{k}{2}} n^k,
\end{align*}
as required.

\subsection{The Hajnal-Szemer\'edi theorem}

We start this section with a brief outline of the proof of Theorem~\ref{thm:sparsehajnalszemeredi}. Fix some $k \geq 3$ and $\HSe > 0$. Given a subgraph $G'$ of $G_{n,p}$ with $\delta(G') \geq (1 - \frac{1}{k} + \HSe)pn$, we apply the regularity lemma to $G'$ to obtain an $(\e,p)$-regular partition of its vertex set into $t$ parts. We then construct an auxiliary graph $R$ on the vertex set $[t]$ whose edges correspond to $(\e,p)$-regular pairs of non-negligible density in the regular partition of $G'$. Since most of the edges of $G'$ lie in such dense and regular pairs, the assumption on the minimum degree of $G'$ implies that $R$ contains an almost spanning subgraph $R'$ with minimum degree at least $(1-\frac{1}{k})t$. By the Hajnal-Szemer{\'e}di theorem, $R'$ contains a $K_k$-factor. Finally, a fairly straightforward application of Theorem~\ref{thm:main} implies that each of the cliques in this $K_k$-factor corresponds to a $K_k$-packing in $G'$ that covers most of the vertices in the $k$ parts of the regular partition that form this clique. As is typically the case with arguments employing the regularity method, the details of the argument are somewhat intricate.

We start the actual proof by fixing several constants. Let
\[
\HSX = \min \left\{ \frac{\HSe}{2}, \frac{1}{2k} \right\}, \quad \HSep = \frac{\HSX^2}{14}, \quad \text{and} \quad t_0 = \left\lceil \frac{1}{\HSep} \right\rceil.
\]
Let $\e'$ be the constant obtained by invoking Proposition~\ref{prop:main} with $d_{\ref{prop:main}} = \HSep$. Let $\e = \min\{\HSX \e', \HSep/2\}$ and $D = 1 + \HSep$. Furthermore, let $T = T_{\ref{prop:sparseregclean}}(\e, D, t_0)$, let $\eta = \min\{\eta_{\ref{prop:sparseregclean}}(\e,D,t_0), 1/T\}$, and let $\eta' = \beta/T$. Assume that $p \geq Cn^{-1/m_2(H)}$ for some large constant $C$ so that a.a.s.\ the random graph $G_{n,p}$ is $(\eta, p, D)$-upper-uniform and every subgraph $G$ of $G_{n,p}$ in $\G\bigl(K_k,n', \bm, p, \e' \bigr)$, where $n' \geq \eta' n$ and $m_{ij} \geq \HSep p (n')^2$ for all $ij \in E(K_k)$, contains a canonical copy of $K_k$. We shall show that, conditioned on the above two events, every $G' \subseteq G_{n,p}$ with $\delta(G') \geq (1 - \frac{1}{k} + \HSe)$ contains a $K_k$-packing covering all but at most $\HSe n$ vertices.

Fix a $G'$ as above and apply Proposition~\ref{prop:sparseregclean} with $d_{\ref{prop:sparseregclean}} = 2\HSep$ to obtain an equipartition $V_1, \ldots, V_t$ of the vertex set of $G'$ into $t_0 \leq t \leq T$ pieces and a subgraph $G''$ of $G'$ all of whose edges lie in $(\e,p)$-regular pairs with edge density at least $2\HSep p$ and which satisfies
\[
e(G') - e(G'') \leq \left(\frac{D}{t_0} + 2D\e + 2\HSep \right)\frac{pn^2}{2} \leq 3\HSep pn^2.
\]
Let $R$ be the graph on the vertex set $[t]$ whose edges are those pairs $ij$ such that the bipartite graph $G''(V_i,V_j)$ is non-empty. Recall that each such bipartite graph is $(\e,p)$-regular and has at least $2 \HSep p |V_i||V_j|$ edges.

\begin{claim} \label{claim:HS}
  The cluster graph $R$ contains a subgraph $R'$ with $\delta(R') \geq (1-\frac{1}{k})t$ and $t' \geq (1-\HSX)t$ vertices for some $t'$ divisible by $k$.
\end{claim}
\begin{proof}
  We construct such a graph $R'$ greedily by sequentially removing from $R$ vertices of degree smaller than $(1-\frac{1}{k})t + k$ and at most $k-1$ further vertices in order to guarantee that $k$ divides $t'$. If this process terminates before we delete from $R$ more than $\HSX t - k$ vertices, then we will arrive at a graph $R'$ with the desired properties. Otherwise, $R$ contains at least $\HSX t - k$ vertices with degree at most $(1-\frac{1}{k} + \HSX)t$. Denote this set by $X$ and observe that
  \begin{align*}
    \sum_{i \in X} \sum_{v \in V_i} \deg_{G'}(v) & \leq \sum_{i \in X} \sum_{v \in V_i} \deg_{G''}(v) + 2(e(G') - e(G'')) = \sum_{i \in X} \sum_{j \neq i} e_{G''}(V_i, V_j) + 2(e(G') - e(G'')) \\
    & \leq |X| \left(1 - \frac{1}{k} + \HSX \right) t \cdot Dp\left(\frac{n}{t}\right)^2 + 6 \HSep p n^2 \leq \frac{|X|}{t} \left(1 - \frac{1}{k} + \HSX + \HSep + \frac{12\HSep}{\HSX}\right)pn^2 \\
    & < \frac{|X|}{t} \left( 1 - \frac{1}{k} + 2 \HSX \right) pn^2.
  \end{align*}
  On the other hand,
  \[
  \sum_{i \in X} \sum_{v \in V_i} \deg_{G'}(v) \geq |X| \cdot \frac{n}{t} \cdot \delta(G') \geq \frac{|X|}{t} \left(1 - \frac{1}{k} + \HSe \right)pn^2,
  \]
  a contradiction, as $\gamma \geq 2\beta$.
\end{proof}

By the Hajnal-Szemer{\'e}di theorem, Theorem~\ref{thm:hajnalszemeredi}, the graph $R'$ contains a $K_k$-factor. Hence, it suffices to show that each subgraph of $G''$ induced by sets $V_{i_1}, \ldots, V_{i_k}$, where $i_1, \ldots, i_k \in [t]$ form a copy $K_k$ in $R$, contains a $K_k$-packing covering at least $(1-\HSX)$-fraction of its vertices. Indeed, since the $K_k$-factor in $R'$ covers at least a $(1-\HSX)$-proportion of all the vertices of the cluster graph $R$ and, as we show below, each of its cliques $i_1, \ldots, i_k$ corresponds to a $K_k$-packing covering at least a $(1-\HSX)$-proportion of the vertices in $V_{i_1} \cup \ldots \cup V_{i_k}$, Theorem~\ref{thm:sparsehajnalszemeredi} will easily follow, as $(1-\HSX)^2 > 1 - \HSe$.

\begin{claim}
  Suppose that $i_1, \ldots, i_k \in [t]$ induce a copy of $K_k$ in $R$. Then the graph $G''[V_{i_1} \cup \ldots \cup V_{i_k}]$ contains a $K_k$-packing covering all but at most $\HSX \frac{n}{t}$ vertices in each of $V_{i_1}, \ldots, V_{i_k}$.
\end{claim}
\begin{proof}
  For simplicity, assume that $i_1 = 1, \ldots, i_k = k$. It will be enough to show that for every choice of $W_1 \subseteq V_1, \ldots, W_k \subseteq V_k$ with $n' = |W_1| = \ldots = |W_k| \geq \HSX \frac{n}{t}$, the graph $G''[W_1 \cup \ldots \cup W_k]$ contains a canonical copy of $K_k$. To this end, observe that this graph belongs to $\G(K_k,n', \bm, p, \e')$ for some $\bm = (m_{ij})_{ij \in E(K_k)}$ with $m_{ij} \geq \HSep p(n')^2$ for all $ij \in E(K_k)$. Indeed, since for each $ij \in E(K_k)$, the graph $G''(V_i, V_j)$ is $(\e,p)$-regular and has at least $2 \HSep p |V_i| |V_j|$ edges, it follows that the graph $G''(W_i, W_j)$ is $(\e',p)$-regular and has at least $\HSep p|V_i||V_j|$ edges. The claim now follows.
\end{proof}

\subsection{The Andr\'asfai-Erd\H{o}s-S\'os theorem}

Fix a graph $H$ and $\gamma > 0$. We start by fixing several constants. Choose $t_1$ so that, for any $t \geq t_1$, the Andr\'asfai-Erd\H{o}s-S\'os theorem \cite{AES74} (or rather its generalization due to Alon and Sudakov \cite{AS06}) holds in the sense that any $H$-free graph on $t$ vertices with minimum degree at least $\left(1 - \frac{3}{3\chi(H) - 4} + \frac{\gamma}{2}\right) t$ may be made $(\chi(H) - 1)$-partite by removing at most $\frac{\gamma}{2} t^2$ edges. Let
\[
\HSX = \frac{\HSe}{4}, \quad \HSep = \frac{\HSX^2}{14}, \quad \text{and} \quad t_0 = \max\{2t_1, \left\lceil 1/\HSep \right\rceil\}.
\]
Let $\e'$ be the constant obtained by invoking Proposition~\ref{prop:main} with $d_{\ref{prop:main}} = \HSep$. Let $\e = \min\{\e', \HSep/2\}$ and $D = 1 + \HSep$. Furthermore, let $T = T_{\ref{prop:sparseregclean}}(\e, D, t_0)$ and $\eta = \min\{\eta_{\ref{prop:sparseregclean}}(\e,D,t_0), 1/T\}$. Assume that $p \geq Cn^{-1/m_2(H)}$ for some large constant $C$ so that a.a.s.\ the random graph $G_{n,p}$ is $(\eta, p, D)$-upper uniform and every subgraph $G$ of $G_{n,p}$ in $\G\bigl(H,n', \bm, p, \e\bigr)$, where $n' \geq \eta n$ and $m_{ij} \geq \HSep p (n')^2$ for all $ij \in E(H)$, contains a canonical copy of $H$. We shall show that, conditioned on the above two events, every $H$-free subgraph $G' \subseteq G_{n,p}$ with $\delta(G') \geq (1 - \frac{3}{3 \chi(H) - 4} + \HSe)pn$ may be made $(\chi(H)-1)$-partite by removing at most $\HSe p n^2$ edges.

Fix a $G'$ as above and apply Proposition~\ref{prop:sparseregclean} with $d_{\ref{prop:sparseregclean}} = \HSep$ to obtain an equipartition $V_1, \ldots, V_t$ of the vertex set of $G'$ into $t_0 \leq t \leq T$ pieces and a subgraph $G''$ of $G'$ all of whose edges lie in $(\e,p)$-regular pairs with edge density at least $\HSep p$ and which satisfies
\[
e(G') - e(G'') \leq \left(\frac{D}{t_0} + 2D\e + \HSep \right)\frac{pn^2}{2} \leq 3\HSep pn^2.
\]
Let $R$ be the graph on vertex set $[t]$ whose edges are those pairs $ij$ such that the bipartite graph $G''(V_i,V_j)$ is non-empty. Recall that each such bipartite graph is $(\e,p)$-regular and has at least $\HSep p |V_i||V_j|$ edges. The following claim is proved in the same way as Claim~\ref{claim:HS}.

\begin{claim}
  The cluster graph $R$ contains a subgraph $R'$ with $\delta(R') \geq (1-\frac{3}{3\chi(H) - 4} +\frac{\HSe}{2})t$ and $t' \geq (1-\HSX)t$ vertices.
\end{claim}

Suppose now that $R'$ contained a copy of $H$. Then there would be pairwise disjoint sets $V_1, \ldots, V_k$ and a sequence $\bm = (m_{ij})_{ij \in E(H)}$ with $m_{ij} \geq \HSep p(\frac{n}{t})^2$ such that $G'[V_1 \cup \ldots \cup V_k] \in \G(H,\frac{n}{t},\bm,p,\e)$. Consequently, by Proposition~\ref{prop:main} and the choice of $\e$, $\eta$, and $C$, the number of copies of $H$ in $G''$ would be positive,
contradicting the assumption that $G'$ was $H$-free.

Therefore, $R'$ is $H$-free and by the choice of $t_0$, the graph $R'$ may be turned into a $(\chi(H) - 1)$-partite graph $R_0$ by removing at most $\frac{\HSe}{2} t'^2$ edges. The corresponding sparse graph $G_0$ whose edges consist of all those edges contained within an edge of $R_0$ is also $(\chi(H) - 1)$-partite. It is obtained from $G'$ by first deleting at most $3 \gamma' p n^2$ edges to form $G''$, then at most $\beta t^2 D p \left(\frac{n}{t}\right)^2$ edges in forming $R'$ from $R$ and, finally, at most $\frac{\gamma}{2} t'^2 D p \left(\frac{n}{t}\right)^2$ edges in forming $R_0$ from $R'$. Overall, this is at most
\[3 \gamma' p n^2+ \beta t^2 D p \left(\frac{n}{t}\right)^2 + \frac{\gamma}{2} t'^2 D p \left(\frac{n}{t}\right)^2 \leq \frac{\gamma}{16} p n^2 + \frac{5\gamma}{16} p n^2 + \frac{9 \gamma}{16} p n^2 < \gamma p n^2.\]
The result follows.


\section{Concluding remarks} \label{sec:conclude}

\subsection{Hypergraphs}

The methods employed in this paper should also extend to work for hypergraphs. That is, one should be able to show that a.a.s~any regular partition of a subgraph of the random hypergraph has a corresponding counting lemma. The main obstacle here is to prove a sparse counterpart to the hypergraph regularity lemma \cite{G06, G07, NRS06, RS04, T06}. This should be a comparatively straightforward hybrid of the hypergraph regularity lemma and the sparse regularity lemma. Once this theorem is in place, either method used in this paper should extend to show that it is effective in the random setting. 

For linear hypergraphs, that is, hypergraphs for which every pair of edges intersect in at most one vertex, the results of this paper generalize more easily. In this case, the relevant regularity and counting lemmas \cite{KNRS10} (see also \cite{CHPS12}) follow in a similar fashion to the usual regularity and counting lemmas. This is because it is enough to have control over edge density on large vertex sets, whereas, for general hypergraphs, we need more elaborate conditions. 

An alternative approach was already used in \cite{CG12} to prove an extension of the hypergraph removal lemma to sparse random hypergraphs. Roughly speaking, rather than applying a sparse regularity lemma, one maps the sparse hypergraph $G$ to its dense model $K$ and then applies the usual hypergraph regularity lemma to this hypergraph $K$. This produces a regular partition which is also regular for the original hypergraph. 

The difference is a matter of quantifiers. If we have a sparse hypergraph regularity lemma then the correct analogue of the K\L R conjecture would be that a.a.s.\ {\it any} regular partition of a subgraph of the random hypergraph has a corresponding counting lemma. This alternative method allows one to say that a.a.s.\ for any subgraph of the random hypergraph there exists {\it some} regular partition for which there is a corresponding counting lemma. Despite this difference, we believe that this method is likely to be sufficient for most applications in the random setting.

\subsection{Removing the need for strict balance}

In Theorem \ref{thm:main} (ii) we assumed that the graph $H$ was strictly balanced, which was sufficient for the applications in this paper. However, the methods of \cite{CG12} can be used to obtain a result without this condition if we allow an extra logarithmic factor. That is, if $p\geq C(\log N)^cN^{-1/{m_2(H)}}$, for $c > 0$ an absolute constant, then a.a.s.\ we have the conclusion 
\[G(H) = (1 \pm \delta) \left(\frac{m}{n^2}\right)^{e(H)} n^{v(H)}\]
that we had in Theorem \ref{thm:main} (ii). We are not formally claiming this as a result, since a certain amount of modification is needed to the proofs in \cite{CG12}, and though this modification appears to be straightforward, we have not written it out in detail. 

However, let us briefly indicate why we are confident that this result is true. It turns out that many of the difficulties involved in proving transference disappear if one allows some extra logarithmic factors in the thresholds. This is because we no longer have to worry about certain large deviations from the expectation. To give an example, recall that the threshold for Tur\'an's theorem for triangles occurs at around $p = n^{-1/2}$. This is the point at which the number of triangles is about the same as the number of edges. So we expect that most edges will be contained in at most a constant number of triangles. However, it will happen that there are some edges which are in many more triangles than expected and this deviation is enough to spoil some of the required estimates. 

To deal with these unwanted deviations, one has to show that they do not occur too often and this adds many extra technicalities to the proof. In \cite{CG12}, the main tool for circumventing these problems was the introduction of so-called capped convolutions. If we allow some extra slack, we no longer have to work with these capped convolutions and this simplifies the proof considerably. In particular, it appears to be straightforward to prove the following result, valid for all graphs $H$.

\begin{claim} \label{thm:transfer2}
There is a positive constant $c$ such that, for any graph $H$ on $k$ vertices and any $\e > 0$, there exist positive constants $C$ and $\lambda$ such that if $C (\log n)^c N^{-1/m_2(H)} \leq p \leq \lambda$ then the following holds a.a.s.\ in the random graph $G_{N,p}$. For all disjoint vertex subsets $V_1, \dots, V_k$ and every subgraph $G$ of $G_{N,p}$ in $H^*(V_1, \dots, V_k)$, there exists a subgraph $K$ of $K_N$ in $H^*(V_1, \dots, V_k)$ such that \begin{equation*}
|p^{-e(H)}\mu^*_H(G) - \mu^*_H(K)| \leq \epsilon
\end{equation*}
and, for all pairs of disjoint vertex subsets $U_1, U_2$ of $V(K_N)$,
\begin{equation*}
|p^{-1} \sum_{x, y} G(x,y) - \sum_{x, y} K(x,y)|\leq \e N^2,
\end{equation*}
where the sums are taken over all $x \in U_1$ and $y \in U_2$.
\end{claim}

The version of Theorem \ref{thm:main} without the strict balance condition is a straightforward consequence of this claim, proved in exactly the same fashion as Part (ii) of Theorem \ref{thm:main}. We omit the details.

\vspace{2mm}
\noindent
{\bf Acknowledgements.} {The three junior authors are indebted to Tibor Szab\'o and his group at the Freie Universit\"at, Berlin for hosting us during part of the period when we were working on this paper. We would also like to thank Jacob Fox for pointing out how to bridge the gap in the sparse removal lemma for balanced graphs.}

\bibliographystyle{abbrv}
\bibliography{references}

\end{document}